\documentclass[11pt,twoside]{article}
\usepackage{simplewick}
\usepackage{times}
\usepackage{amsmath,amssymb,amsthm}
\usepackage{color}
\usepackage{hep}
\usepackage{mathrsfs}
\usepackage{leftidx}
\usepackage{graphicx}
\usepackage{float}
\usepackage{enumerate}
\usepackage{titletoc}
\usepackage{epstopdf}
\usepackage{subfigure}
\usepackage{ulem}

\allowdisplaybreaks
\def \no{\nonumber}

\def\p{\partial}

\def\f{\frac}
\def\na{\nabla}

\def\al{\alpha}
\def\t{\tilde}

\def\vp{\varphi}
\def\O{\Omega}
\def\th{\theta}
\def\g{\gamma}
\def\G{\Gamma}

\def\dl{\delta}
\def\p{\partial}

\def\f{\frac}
\def\k{\kappa}

\def\na{\nabla}

\def\al{\alpha}

\def\t{\tilde}
\def\o{\omega}
\def\O{\Omega}
\def\vp{\varphi}
\def\th{\theta}

\def\g{\gamma}
\def\G{\Gamma}

\def\dl{\delta}

\def\ds{\displaystyle}
\allowdisplaybreaks

\begin{document}
	\footskip=0pt
	\footnotesep=2pt
	\let\oldsection\section
	\renewcommand\section{\setcounter{equation}{0}\oldsection}
	\renewcommand\thesection{\arabic{section}}
	\renewcommand\theequation{\thesection.\arabic{equation}}
	\newtheorem{claim}{\noindent Claim}[section]
	\newtheorem{theorem}{\noindent Theorem}[section]
	\newtheorem{lemma}{\noindent Lemma}[section]
	\newtheorem{proposition}{\noindent Proposition}[section]
	\newtheorem{definition}{\noindent Definition}[section]
	\newtheorem{remark}{\noindent Remark}[section]
	\newtheorem{corollary}{\noindent Corollary}[section]
	\newtheorem{example}{\noindent Example}[section]
	
	\title{Global smooth solutions to 2D
		semilinear wave equations \\ with large data}
	\author{Bingbing Ding\thanks{School of Mathematical Sciences and Mathematical Institute, Nanjing Normal University, Nanjing 210023, China. Email: 13851929236@163.com, bbding@njnu.edu.cn}\and Shijie Dong\thanks{Southern University of Science and Technology,
			Shenzhen International Center for
			Mathematics, and Department of Mathematics, Shenzhen 518055, China. E-mail: dongsj@sustech.edu.cn, shijiedong1991@hotmail.com}
		\and Gang Xu\thanks{School of Mathematical Sciences and Mathematical Institute, Nanjing Normal University, Nanjing 210023, China. E-mail: gxumath@outlook.com, gxu@njnu.edu.cn}}
	\date{}
	\maketitle

	\begin{abstract}		
		We are interested in coupled semi-linear wave equations satisfying the null condition in two space dimensions, a basic model in nonlinear wave equations. Our aim is to establish global existence of smooth solutions to this system with large initial data of short pulse type. Major difficulties arise due to the largeness of initial data and the slow decay nature of 2D wave equations. To overcome the difficulties, by careful examination of the local solutions, we adapt various vector-field methods to different spacetime regions with several novel weighted energy estimates.

	\end{abstract}
	
	\tableofcontents

	\vskip 0.2 true cm {\bf Keywords:} 2D semilinear wave equation, short pulse initial data,
	null condition

	\vskip 0.2 true cm {\bf Mathematical Subject Classification:} 35L05, 35L72

	\vskip 0.4 true cm
	
	\section{Introduction}\label{in}
	\subsection{Model problem and main result}\label{2}
	The system of semi-linear wave equations we consider is of the form
	\begin{equation}\label{semi}
	\Box\phi^I=Q^I(\p\phi,\p\phi),\qquad I=1,\cdots, N,
	\end{equation}
	where $\phi=(\phi^1,\cdots,\phi^N)$ is a vector valued function,  $\Box=-\p_t^2+\triangle$, $t=x^0\in[1,\infty)$, $x=(x^1,x^2)$, and $\p=(\p_0,\p_1,\p_2)$. In addition, for any $I\in\{1, \cdots N\}$, $Q^I(\p\phi, \p\phi)$ is a quadratic form satisfying the null conditions, that is,
	\begin{equation}\label{Q}
	Q^I(\p\phi, \p\phi)=\sum_{\tiny\begin{array}{c}1\leq A,B\leq N\\0\leq\al,\beta\leq 2\end{array}}g_{AB}^{\al\beta,I}\p_\al\phi^A\p_\beta\phi^B,
	\end{equation}
	in which $g_{AB}^{\al\beta,I}$ are constants satisfying
	\begin{equation}\label{null}
	\sum_{0\leq\al,\beta\leq 2}g_{AB}^{\al\beta,I}\xi_\al\xi_\beta\equiv0,\ \text{on}\ \xi_0^2=\xi_1^2+\xi_2^2.
	\end{equation}
	
	In the present paper, we are concerned about global existence of solution to \eqref{semi}-\eqref{null} with large data of short pulse type, that is
	\begin{equation}\label{initial}
	\left\{
	\begin{aligned}
	&\phi(1,x)=\dl^{1-\kappa}\phi^\dl_0(\f{r-1}\dl,\o),\\
	&\p_t\phi(1,x)=\dl^{-\kappa}\phi^\dl_1(\f{r-1}\dl,\o),
	\end{aligned}
	\right.
	\end{equation}
	where
	$r=|x|=\sqrt{(x^1)^2+(x^2)^2}$, $\o=(\o_1,\o_2)=\f{x}{r}\in\mathbb S^1$, $\kappa\in(0,\kappa_0)$ with $\kappa_0={\f12}$ is a fixed constant, and $(\phi^\dl_0,\phi^\dl_1)(s,\o)$ are smooth functions defined
	in $\mathbb R\times \mathbb S^1$
	compactly supported in $(-1,0)\times \mathbb S^1$.
	
	\begin{remark}
		The initial data \eqref{initial} referred to as ``short pulse data" were introduced in the seminal work \cite{Christodoulou09} by Christodoulou. Though $\phi(1,x)$ is small due to the smallness of $\delta$, it becomes significantly large after taking derivatives. Furthermore, the more derivatives we take, the larger the data become, that is,
		\begin{equation}\label{pphi}
		\p_{x}^\al\phi(1,x)=O(\dl^{1-\kappa-|\al|}).
		\end{equation}

		
	\end{remark}
	To show global existence to the Cauchy problem \eqref{semi}-\eqref{initial}, we further assume
	\begin{equation}\label{inital}
	(\p_t+\p_r)\O^j\p^q\phi|_{t=1}= O(\delta^{1-\kappa-|q|}),
	\end{equation}
	in which $\O=x^1\p_2-x^2\p_1$ is the rotational derivative on $\mathbb S^1$.
	This assumption on the initial data is rather weak compared with some existing literature; see Remark \ref{rmk:largeness} for more details.
	
	\begin{remark}
		There are plenty of initial data satisfying \eqref{initial}-\eqref{inital}. For example, for arbitrary smooth functions $\psi_1(s,\o)$ and $\psi_2(s,\o)$
		compactly supported in $(-1,0)\times \mathbb S^1$, it is easy to check that
		\[
		\phi_0^\dl(s,\o)=\psi_1(s,\o)\quad\text{and} \quad\phi_1^\dl(s,\o)=-\p_s\psi_1(s,\o)+\dl\psi_2(s,\o)
		\]
		with $s = {r-1\over \delta}$ fulfill \eqref{pphi}-\eqref{inital}.
	\end{remark}

	We are now ready to state the main result in this paper.
	\begin{theorem}\label{main}
		Consider the system \eqref{semi} under null condition \eqref{null}, and let $0<\kappa< \kappa_0$. Then, there exists a $\delta_0 > 0$ such that for all $\delta \in (0, \delta_0)$ and all initial data obeying \eqref{inital}, the Cauchy problem \eqref{semi}-\eqref{initial} admits a global smooth solution $\phi$ which satisfies
		$$\phi\in C^\infty([1,+\infty)\times\mathbb R^2),\quad\qquad  |\p\phi|\le C\delta^{-\kappa}t^{-1/2},$$
		for all time $t\ge 1$, where $C>0$ is a uniform constant independent of $\dl$ and $\kappa$.
	\end{theorem}
	
	\begin{remark}\label{rmk:largeness}
		The  ``short pulse data"  have been studied  in \cite{MPY, Wang} etc., but the class of data considered in the present paper  are different from theirs. Recall for instance in \cite{MPY}, the initial data $(\phi, \p_t\phi)|_{t=1}$ there should fulfill a similar restriction to  \eqref{inital}, and more specifically the authors assumed
		\begin{equation}\label{MPYinitial}
		(\p_t+\p_r)^k\O^j\p^q\phi|_{t=1}= O(\delta^{1/2-|q|}),\qquad \forall \,\, k\leq N
		\end{equation}
		for a large integer $N$. \eqref{MPYinitial} guarantees the existence of solutions in the region away from the outermost characteristic cone to be a small data 3D problem, thus the largeness of solutions only reflects near the outermost cone.  However, in our paper, the weaker condition \eqref{inital} makes the energies be large both near the outermost cone $\{t=r\}$ and inside it, which causes additional difficulties.
	\end{remark}
	
	\begin{remark}
		Despite system \eqref{semi} satisfies the null condition \eqref{null}, it does not seem easy to apply the classical energy method to deduce the global existence even if we treat small initial data since the decay rate of the solution in 2D is too slow. Thus, it is crucial to choose applicable multipliers to establish energies suitable to our problem. It is  more challenging to treat the 2D case than the 3D case.
	\end{remark}

	\subsection{Brief history and relevant results}
	
	The study of nonlinear wave equations has been an active research field for decades. One problem of fundamental importance is that whether a nonlinear wave system admits a global solution with certain assumptions on the initial data. In 3D,  John in \cite{John} showed that general quadratic nonlinearities would lead to finite time blow-up for wave equations even for small, smooth initial data. In 1986, thanks to the seminal works of Klainerman \cite{Klainerman86} and Christodoulou \cite{Christodoulou86}, it is known that small global solutions exist if the nonlinear terms satisfy null condition.
	
	Later on, in the breakthrough \cite{Alinhac2001} Alinhac proved that quasilinear wave equations under null condition in 2D admits global solutions for small, smooth and compactly supported initial data, where the celebrated idea of ghost weight energy estimates came up. For this 2D quasilinear null model, Hou-Yin \cite{Hou-Yin-1} removed the compactness assumption on the initial data and more recently Dong-LeFloch-Lei \cite{DLL} and Li \cite{Li23} independently established a uniform boundedness result for the energy. We also recall that global existence for the 2D semilinear wave equations (i.e., \eqref{semi}) with small data was proved by Katayama \cite{Katayama} where the author used ghost weight method; see also \cite{Wong}.
	
	For the one-dimensional case, we would like to mention the global existence results by Gu \cite{Gu} on wave maps and by Luli-Yang-Yu \cite{LYY18} on the system \eqref{semi}.
	
	Following those small data results, we now turn to the literature for nonlinear wave equations with large initial data. The short pulse data, introduced by Christodoulou, are an important class of large data. In 2009, Christodoulou introduced the idea of short pulse data in \cite{Christodoulou09} when studying the Einstein vacuum equations, and showed the formation of the trapped surface. Motivated by this milestone work, various studies on nonlinear wave equations with short pulse initial data are conducted, including both global existence and finite time blow-up results. In \cite{MPY},  Miao-Pei-Yu proved global existence for \eqref{semi} with short pulse data in 3D. In \cite{MY}, Miao-Yu showed formation of shocks for a class of quasilinear wave equations with short pulse data in 3D, and interestingly the quasilinear wave equations admit global solutions when the initial data are of small size. In \cite{Ding4}, Ding-Xin-Yin got global existence of solutions for general quasilinear wave equations with short pulse data in 4D case. For 3D case, when $\kappa$ is in different range, Ding-Lu-Yin proved global existence for a class of quasilinear wave equations with short pulse data in \cite{Ding3} while the shock formation was shown by Lu-Yin \cite{Lu1}. In \cite{Wang}, Wang-Wei obtained global existence for 2D relativistic membrane equations while Ding-Xin-Yin in \cite{Ding3} established global existence for 2D isentropic and irrotational Chaplygin gases. We also list the global existence result for 3D nonlinear wave equations with large data  not of short pulse form by Yang \cite{Yang} and Luk-Oh-Yang \cite{LOY}.

	\subsection{Major challenges}
	
	For small initial data, the most well-known challenge in studying \eqref{semi} is the slow decay nature of waves in 2D. We recall that free waves decay at speed $t^{-(d-1)/2}$ in time in $\mathbb{R}^{d+1}$ with $d\geq 1$. In 2D, the best we can expect for the system \eqref{semi} is that the solution decays at speed $t^{-1/2}$, which is a non-integrable quantity (actually $t^{-1/2}$ is even far from the boarline of non-integrability $t^{-1}$). Thus, it is already a non-trivial task to prove global existence for system \eqref{semi} with small, smooth initial data.
	To show Theorem \ref{main}, we additionally need to treat large initial data of short pulse form (with few smallness assumptions), which further bring severe difficulties.
	
	First, the solution to \eqref{semi}-\eqref{inital} exhibits different levels of largeness when hit with different weighted derivatives.  The largeness of the solution is reflected by the parameter $\delta$ with negative power (recall $\delta$ is small); say $\delta^{-1}, \delta^{-2}$, etc.   Short pulse initial data have long range effect, and solutions will stay large as time evolves. Unlike the small initial data case, we now need to carefully track the largeness of the solution hit with different weighted derivatives which is sensitive to our analysis.
	
	Besides, the solution $\phi$ exhibits different levels of largeness in different spacetime regions, which can be found when we analyze the properties on the local solution (see \eqref{local1-2}-\eqref{local2-2}). This does not allow us to treat the solution using a unified way. In addition to consider how to choose targeted methods to conquer the difficulties causing by largeness and slow decay in different regions, we also pay special attention to the interface of the regions. In fact, we will face an initial boundary value problem. The presence of the boundary with non-zero data will force us to analyze the terms on the boundary carefully (see the proof of Lemma \ref{energy}), which further increases the difficulties of showing global existence for \eqref{semi}.

	\subsection{Novel ideas and outline of the proof}\label{int:1.4}
	
	\paragraph{A glimpse of example}	
	
	\
	
	To figure out how the ``largeness" will develop in the equation, before starting our proof for general system of semi-linear equations in 2D, we take a glance at a scalar equation, which reads $\Box\bar\phi=|\p_t\bar\phi|^2-|\nabla_x\bar\phi|^2$
	with $x\in\mathbb R$. We perform the Nirenberg transformation\footnote{This transformation, however, cannot be applied to our problem, as our problem deals with coupled wave equations.} $\tilde\phi=1-e^{-\bar\phi}$, and then $\tilde\phi$ solves $\Box\tilde\phi=0$.
	Specifically, we first analyze the local $L^\infty$ properties of
	\begin{equation}\label{swequ}
	\left\{
	\begin{aligned}
	&\p_t^2\tilde\phi-\p_x^2\tilde\phi=0,\qquad (t,x)\in (1,\infty)\times\mathbb R,\\
	&\t\phi(1,x)=g(x)=\dl^{1-\kappa}\t\phi^\dl_0(\f{|x|-1}\dl),\qquad x\in\mathbb R,\\
	&\p_t\t\phi(1,x)=h(x)=\dl^{-\kappa}\t\phi^\dl_1(\f{|x|-1}\dl),\qquad x\in\mathbb R,
	\end{aligned}
	\right.
	\end{equation}
	where $\kappa$ is a constant as before, and $(\t\phi^\dl_0, \t\phi^\dl_1)(s)$ are smooth functions defined in $\mathbb R$ compactly supported in $(-1,0)$ with the same estimates as in \eqref{inital}. For simplicity, we only focus on the region $D_0=\{(t,x):1\leq t\leq1+2\dl,2-\dl-t\leq x\leq t\}$. In this region,
	\begin{equation*}
	\t\phi(t,x)=\f12\big(g(x+t-1)+g(x-t+1)\big)+\f12\int_{x-t+1}^{x+t-1}h(y)dy,
	\end{equation*}
	and this, together with \eqref{inital}, gives
	\begin{equation}\label{tphi}
	|\p^\al\t\phi(t,x)|+|(\p_t+\p_x)\p^\al\t\phi(t,x)|\lesssim\delta^{1-|\al|-\kappa},\quad (t,x)\in  D_0,
	\end{equation}
	with $\p=(\p_t,\p_x)$ here. Furthermore, it follows from $(\p_t+\p_x)(\p_t-\p_x)^k\p^\al\t\phi=(\p_t-\p_x)(\p_t+\p_x)^k\p^\al\t\phi=0$ for any integer $k\geq 1$ and $\t\phi(t,x)$ vanishes when $t=x$ or $t+x=2-\dl$ that
	\begin{align}
	&(\p_t-\p_x)^k\p^\al\t\phi(t,x)=0,\ (t,x)\in D_0\ \text{and}\ t-x\geq\dl,\label{tphileft}\\
	&(\p_t+\p_x)^k\p^\al\t\phi(t,x)=0,\ (t,x)\in D_0\ \text{and}\ t+x\geq 2\label{tphiright}
	\end{align}
	after integrating along integral curves of $\p_t+\p_x$ and $\p_t+\p_x$ respectively.
	
	Through this toy model, we detect extra smallness properties of the solution in different regions listed in \eqref{tphileft} and \eqref{tphiright}. This guides us how to divide the regions and treat them using different methods in our 2D problem; see Section \ref{LE} for more details.
	
	\paragraph{Outline of the proof.}	
	
	\
	
	To prove Theorem \ref{main}, in addition to study the local properties as above, we also need to overcome the difficulties caused by slow decay nature of 2D waves and by the presence of large initial data (of short pulse type) when $t$ is large. Thus to apply Klainerman's (hyperboloidal) vector-field method, new ingredients and non-trivial techniques should be engineered. We establish new estimates for waves on various types of surfaces, and design carefully chosen bootstraps that are then closed. A broad overview of the proof for Theorem \ref{main} is outlined below with key strategies detailed by sections.
	
	\begin{figure}[htbp]
		\centering
		\includegraphics[scale=0.7]{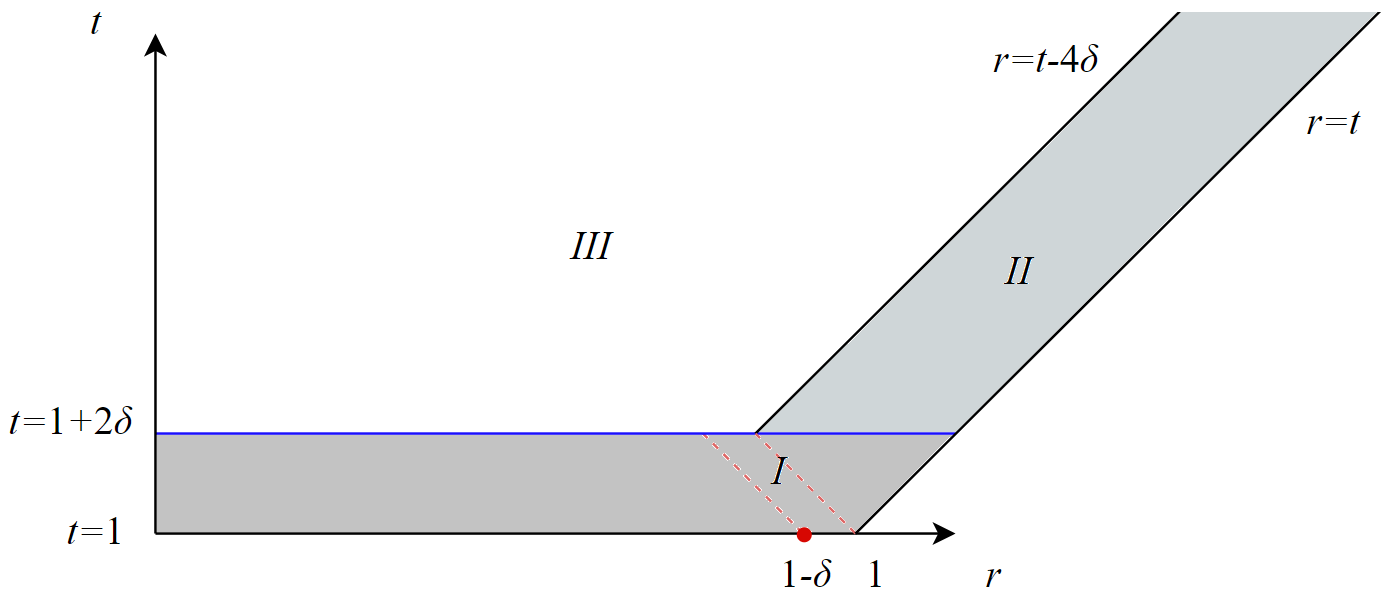}
		\caption{Progressive resolution within three decomposed regions}\label{pic:p4}
	\end{figure}
	
	\vskip 0.2 true cm

	\textbf{Part I: Local existence of $\phi$.} The initial data for \eqref{semi} are set at $t=1$ with support in $\{1-\delta\leq r\leq 1\}$, and we first establish local existence for the solution $\phi$ till $t=t_0=1+2\delta$ with several estimates  that track the smallness/largeness of the solution (inspired by \eqref{swequ}-\eqref{tphiright}). These estimates tell us some directional derivatives (say $\partial_t \pm \partial_r$) stay small in certain regions (see Theorem \ref{Th2.1} for more details), and these are vital to guide us to divide the whole region into one near the cone and an interior one, where the methods used to treat \eqref{semi} are quite different.
	
	\textbf{Part II. Global existence near the cone $\{r=t\}$.}  We expect the same smallness/largeness captured in the local existence part can be kept near the light cone $\{r=t\}$. If we do the estimates for the solution in a unified way in the whole region $\{ t\geq t_0, r\leq t  \}$, we will lose some smallness description of the solution in different regions. Our strategy is to first establish global existence of $\phi$ in the spacetime region $\{ 0\leq t-r \leq 4\delta \}$ (i.e., near the outermost outgoing cone) with control of the solution in this region.
	
	To show global existence in this region, our strategy is to first build estimates to control powers of $\delta^{-1}$ with a slow time growth of the energies, and then to remove this time growth via a new set of energy estimates with the same control of $\delta^{-1}$ power kept.


	\textbf{Part III: Global existence in the interior region.}
	To treat the interior region of the light cone $\{ r\leq t-2\delta \}$, due to the slow decay rate of the 2D wave solutions, we adopt the hyperboloidal vector-field method\footnote{This method was introduced by Klainerman in the Klein-Gordon context \cite{Klainerman85, Klainerman93} for small initial data, and was further developed by Hormander \cite{Hormander}, Psarelli \cite{Psarelli}, and more recently by LeFloch-Ma \cite{LM}, Klainerman-Wang-Yang \cite{SK-QW-SY} as well as many others.}, i.e., the vector-field method on hyperboloids.  In this paper, we extend this method to an initial boundary value problem with large data, and we use the hyperbolic time $\tau=\sqrt{(t+1)^2-r^2}$ to foliate the interior of the light cone. Taking advantage of this special foliation of the spacetime, this method allows one to benefit from the $(t-r)$-decay of the wave solution, and in addition, one is able to derive almost sharp decay of the wave solutions.
	
	In our setting, we need to solve an initial boundary value problem using the hyperboloidal method. In an analogous way to earlier works, we re-establish the (conformal) energy estimates, Klainerman-Sobolev inequality as well as other tools. We design a bootstrap setting that balances $\delta$-dependence and decay rate of the solution, and finally close it with various delicate estimates.

	\subsection{Organisation of the paper}
	
	In Section \ref{sec:notation}, we introduce pertinent notation. In Section \ref{LE}, we build a local existence result with several important estimates for late use. In Section \ref{YY}, we focus on global existence of the solution in the region near the cone $\{ r=t \}$. Finally, we prove global existence of the solution in the interior region and thus for Theorem \ref{main} in Section \ref{sec:interior}.

	\section{Preliminaries}\label{sec:notation}

	We first  introduce the some conventional notation. We set
	\begin{align*}
	&\o_0=-1,\\
	&\o^i:=\o_i=\f{x^i}{r},\quad i=1,2,\\
	&\o_\perp:=(-\o^2, \o^1).
	\end{align*}
	We use the null coordinates
	$$
	u=\f12(t-r),
	\qquad\qquad\qquad
	\underline u=\f12(t+r),
	$$
	and the null frame
	$$
	L:=\p_t+\p_r,
	\qquad\qquad\qquad
	\underline L:=\p_t-\p_r.
	$$	
	The usual vector fields include
	\begin{align*}
	\text{rotation} \,\,\,\,  \O&:= x^1\p_2 - x^2\p_1, \\
	\text{scaling} \,\,\,\,  S&:=t\p_t+r\p_r=\f{t-r}{2}\underline L+\f{t+r}{2}L, \\
	\text{Lorentz boosts} \,\,\,\,  H_i&:=t\p_i+x^i\p_t=\o^i\big(\f{r-t}{2}\underline L+\f{t+r}{2}L\big)+\f{t\o^i_\perp}{r}\O,\\
	\text{good derivatives} \,\,\,\,  T_i & :=\p_i+\o_i\p_t=\o^iL+\f{1}{r}\o^i_\perp\O.
	\end{align*}
	We denote $t_0 = 1+2\delta$ while the initial data are posed at $t=1$, and $\Sigma_{t}:=\{(s,x): s=t, x\in\mathbb R^2\}$.

	Let $m=(m_{\al\beta})=(m^{\al\beta})=diag(-1,1,1)$ denote the Minkowski metric, then $\Box=m^{\al\beta}\p_{\al\beta}^2$, here and throughout the whole paper, Einstein’s summation convention is used.

	For the positive quantities $f$ and $g$, $f\lesssim g$ means $f\leq Cg$ with generic positive constant $C$
	which is independent of $t, x$ and $\dl$.

	For any constants $g^{\al\beta}$ satisfying $g^{\al\beta}\xi_\al\xi_\beta\equiv0$ on ${\xi_0}^2={\xi_1}^2+{\xi_2}^2$, there exist some constants $G_1$ and $G_2^{\al\beta}$ such that
	\begin{equation}\label{QI}
	g^{\al\beta}(\p_\al\vp)(\p_\beta\psi)=G_1[(\p_t\vp)(\p_t\psi)-\nabla\vp\cdot\nabla\psi]+G_2^{\al\beta}[(\p_\al\vp)(\p_\beta\psi)-(\p_\al\psi)(\p_\beta\vp)],
	\end{equation}
	where $\nabla\varphi=(\p_1\varphi,\p_2\varphi)$ and $\nabla\varphi\cdot\nabla\psi=\sum_{i=1}^2(\p_i\varphi)(\p_i\psi)$. Thus,
	\begin{equation}\label{cnull}
	|g^{\al\beta}(\p_\al \vp)(\p_\beta \psi)|\lesssim |T\vp||\p \psi|+|\p \vp||T \psi|,
	\end{equation}
	where $|T\vp|=|T_1\vp|+|T_2\vp|$.

	\section{Local existence of the smooth solution $\phi$}\label{LE}
	
	In this section, we utilize the energy method
	to establish the local existence of the smooth solution $\phi$ to
	equation \eqref{semi} with \eqref{null} and \eqref{inital} for $1\leq t\leq t_0$, meanwhile, several
	key estimates of $\phi(t_0, x)$ on some special space domains are derived.

	\begin{lemma}	[Local existence and basic $L^\infty$ estimates.] \label{lem:local}
		Under the assumption \eqref{inital}, when $\delta>0$ is small, equation \eqref{semi} with \eqref{inital}
		admits a local smooth solution $\phi\in C^\infty([1, t_0]\times\mathbb R^2)$, which satisfies
		\begin{align}
		|L^a\p^\al\O^c\phi(t,x)|\lesssim\delta^{1-|\al|-\kappa},
		\end{align}
		with $a\leq 1$.
	\end{lemma}
	\begin{proof}
		Denote $Z_g$ by any fixed vector filed in $\{S, H_i, i=1,2\}$. Suppose that for $1\leq t\leq t_0$, $\nu\in \mathbb N_0^3$
		and $N_0\in\mathbb N_0$ with $N_0\ge 3$,
		\begin{equation}\label{ba}
		|\p^\al\O^cZ_g^\nu\phi|\leq\delta^{\f12-|\al|}\quad (|\al|+c+|\nu|\leq N_0,\quad |\nu|\leq 1).
		\end{equation}
		We define the following energies for $\phi$
		$$
		M_n(t):=\sum_{|\al|+c+|\nu|\leq n}\delta^{2|\al|}\|\p\p^\al\O^cZ_g^\nu\phi(t,\cdot)\|_{L^2(\mathbb R^2)}^2
		$$
		for $n\leq 2N_0-2$ and $|\nu|\leq 1$.
		Set $w^I=\delta^{|\al|}\p^\al\O^cZ_g^\nu\phi^I$ with $|\al|+c+|\nu|\leq 2N_0-2$. Then it follows from equation \eqref{semi}
		and direct computations that
		\begin{equation}\label{int}
		\int_{\Sigma_t}|\p w^I(t, x)|^2dx=\int_{\Sigma_1}|\p w^I(1, x)|^2dx-2\int_1^t\int_{\Sigma_\tau}(\p_\tau w^I\Box w^I)(\tau,x)dxd\tau
		\end{equation}
		with
		\begin{equation}\label{bgw}
		\Box w^I=\dl^{|\al|}\sum_{|\beta|\leq|\nu|}C_\beta^\nu\p^\al\O^c Z_g^\beta\Big(Q^I(\p\phi,\p\phi)\Big)\\
		\end{equation}
		where $C_\beta^\nu$ are constants depending on $\beta$ and $\nu$.
		It follows from \eqref{ba} that $$\ds|\Box w^I|\lesssim\sum_{\tiny\begin{array}{c}
			|\al_1|\leq|\al|, c_1\leq c,\\|\nu_1|\leq|\nu|\end{array}}\dl^{-\f12+|\al_1|}|\p\p^{\al_1}\O^{c_1}Z_g^{\nu_1}\phi|.$$
		Therefore,
		\begin{equation}\label{ew}
		\begin{split}
		&\int_{\Sigma_t}|\p w^I(t, x)|^2dx
		\lesssim\int_{\Sigma_1}|\p w^I(1, x)|^2dx+\dl^{-\f12}\int_1^tM_n(\tau)d\tau
		\end{split}
		\end{equation}
		by \eqref{int}, which implies
		$$
		M_{2N_0-2}(t)\lesssim M_{2N_0-2}(1)\lesssim\delta^{1-2\kappa}\ \ \text{for}\ t\in[1,t_0]
		$$
		with the help of Gronwall's inequality.
		
		We next close the bootstrap assumption \eqref{ba}. By the following Sobolev's imbedding theorem on the circle $\mathbb S^1_r$
		(with center at the origin and radius $r$):
		$$
		|w^I(t,x)|\lesssim \f{1}{\sqrt{r}}\|\O^{\leq 1}w^I\|_{L^2(\mathbb S^1_r)},
		$$
		together with $r\sim 1$ for $t\in [1,t_0]$ and $(t,x)\in \textrm{supp}\ w^I$, one then has
		\begin{equation}\label{le}
		|\p^\al\O^c Z_g^\nu\phi(t,x)|\lesssim \|\O^{\leq 1}\p^\al\O^c Z_g^\nu\phi\|_{L^2(\mathbb S^1_r)}\lesssim\delta^{1/2}\|\p\O^{\leq 1}\p^\al\O^c Z_g^\nu\phi\|_{L^2(\Sigma_{t})}\lesssim\delta^{1-|\al|-\kappa},\quad
		\end{equation}
		when $|\al|+c+|\nu|\leq 2N_0-3$ and $|\nu|\leq 1$. Therefore, \eqref{ba} is closed.
		
		In addition, by $\ds L=\f{S+\o^iH_i}{t+r}$, we arrive at
		\begin{equation}\label{Lle}
		|L^a\p^\al\O^c\phi(t,x)|\lesssim\sum_{|\nu|\leq 1}|Z_g^\nu\p^\al\O^c\phi(t,x)|\lesssim\delta^{1-|\al|-\kappa}
		\end{equation}
		with $|\al|+c+a\leq 2N_0-3$ and $a\leq 1$.
	\end{proof}
	
	We get rough estimates for the local solution $\phi$ within its support in Lemma \ref{lem:local}.  Motivated by the brief discussion in Section \ref{int:1.4}, we expect more detailed estimates of the local solution varying in different spacetime regions, which are stated in the following theorem.
	
	\begin{theorem}[Detailed estiamtes for the local solution] \label{Th2.1}
		Under the same assumptions as Lemma \ref{lem:local}, for any  $m, n,c\in\mathbb N_0$, $\al\in\mathbb N_0^3$, we have
		\begin{enumerate}[(i)]
			\item
			\begin{align}
			&|L^m\p^\al\O^c\phi(t_0,x)|\lesssim\delta^{1-|\al|-\kappa},\quad r \in[1-2\delta, 1+2\delta],\label{local1-2}\\
			&|\underline L^n\p^\al\O^c\phi(t_0,x)|\lesssim\delta^{1-|\al|-\kappa},\quad r\in [1-3\delta, 1+\delta].\label{local2-2}
			\end{align}
			\item
			{\begin{equation}\label{local3-2}
			|\p^\al\O^c\phi(t_0,x)|\lesssim
			\delta^{2-|\al|-\kappa},
			\quad r\in [1-3\delta, 1+\delta].
			\end{equation}}
			\item
			\begin{equation}\label{local3-3}
			|\underline L^nL^m\O^c\phi(t_0,x)|\lesssim\delta^{1-\kappa},\quad r\in[1-2\delta,1+\delta].
			\end{equation}
		\end{enumerate}
	\end{theorem}
	\begin{remark}	
		By Huygens principle, it is easy to know $\phi$ is supported in $r \in [1-3\delta, 1+2\delta]$ at $t=t_0 = 1+2\delta$. Our estimates in Theorem \ref{Th2.1} provide a more precise description of the solution at $t=t_0$, i.e., different weighted derivatives on $\phi$ give different smallness/largeness in different regions, which guides us to treat differently the regions $\{ r\geq t - 4\delta \}$ and $\{r \leq t-2\delta\}$ when $t\geq t_0$.
	\end{remark}	
	
	\begin{remark}
		We note \eqref{local1-2} indicates that  hitting $L$ on $\phi$ will not increase largeness (i.e., negative powers of $\delta$) in the region $\{ r\geq t - 4\delta, t\geq t_0\}$. This is important for the proof in Section \ref{YY}.
	\end{remark}	
	\begin{remark}
		We recall that $\p \phi$ does not have smallness (except $L \phi$) at $t=1$. Very interestingly, $\p \phi$ enjoys smallness in the region $r\in [1-3\delta, 1+\delta]$ at $t=t_0$; see \eqref{local2-2}, \eqref{local3-2}. This is vital for the argument in Section \ref{sec:interior} to work.
	\end{remark}
	\begin{proof}
		We will study two different regions in Figure \ref{pic:p1} and Figure \ref{pic:p2} in the following proof.	
		
		\textbf{Step 1: Proof of \eqref{local1-2}.}
		
		In Lemma \ref{lem:local} we already show \eqref{local1-2} for $m=1$, and now we prove that \eqref{local1-2} also holds for $m\geq 1$.
		\begin{figure}[htbp]
			\centering
			\includegraphics[scale=0.2]{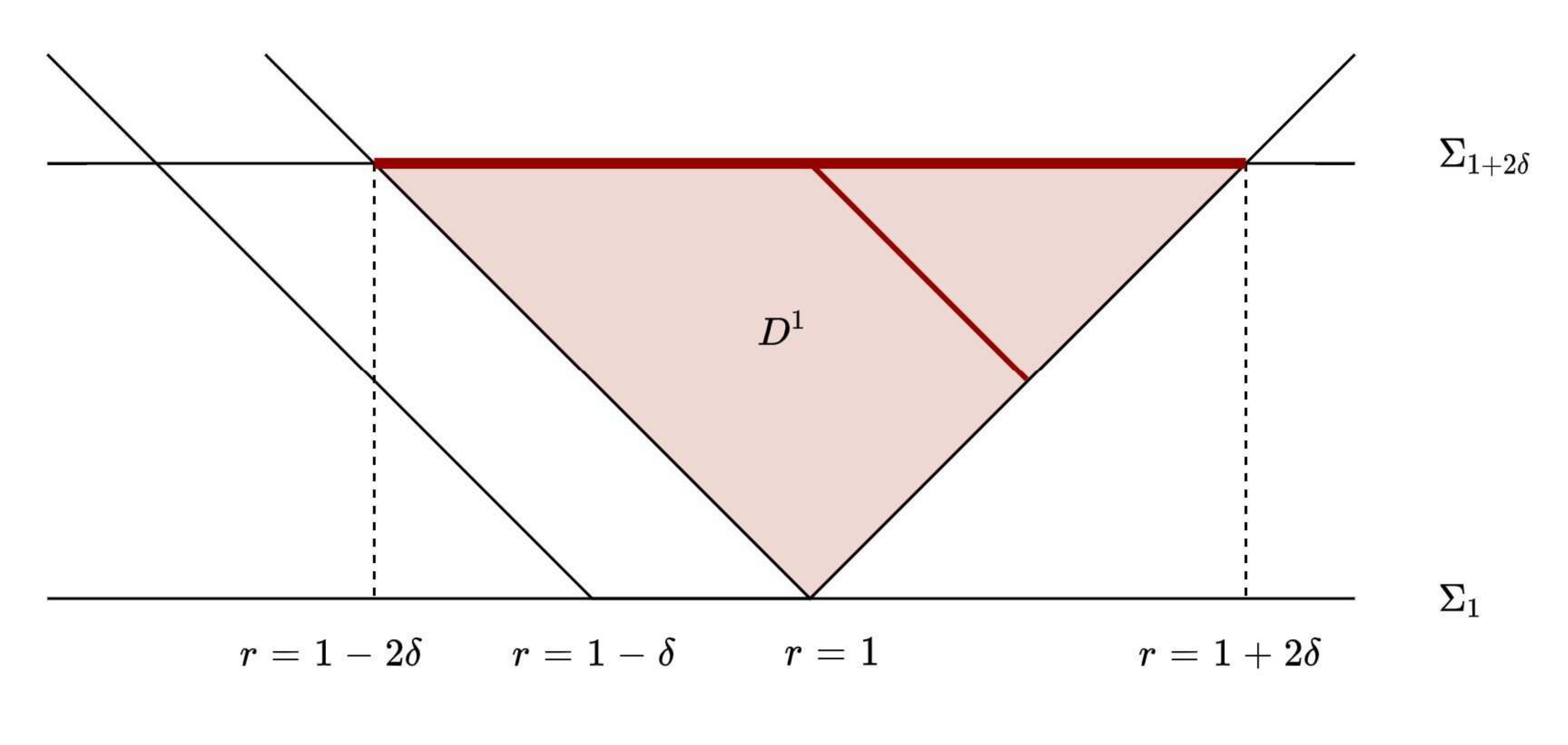}
			\caption{Space-time domain $D^1=\{(t,r): 1\le t\le t_0, 2-t\le r\le t\}$}\label{pic:p1}
		\end{figure}
		Now we start to improve the $L^\infty$ estimate of $\phi(t,x)$ in $D^1$. To this end,
		we rewrite equation \eqref{semi} as
		\begin{equation}\label{me}
		L\underline L\phi^I=\f{1}{2r}L\phi^I-\f{1}{2r}\underline L\phi^I+\f{1}{r^2}\Omega^2\phi^I-Q^I(\p\phi,\p\phi).
		\end{equation}
		We prove \eqref{local1-2} with an induction argument, that is, assume that in $D^1$,
		\begin{equation}\label{Lm}
		|L^m\bar\p^{\iota}\O^c\phi|\lesssim\dl^{1-\kappa-|\iota|}, \ \ \text{for}\ |\iota|+c+2m+1\leq N_0\ \text{and}\ 1\leq m\leq m_0,
		\end{equation}
		where $\bar\p\in\{\p_t,\p_r\}$ and $\iota\in\mathbb N_0^2$. Then, if $|\iota|+c+2m_0+2\leq N_0$,
		acting the operator $L^{m_0}\bar\p^{\iota}\O^c$ on both sides of \eqref{me}, then using \eqref{Lm} and \eqref{cnull} to get
		\begin{equation*}
		\delta^{|\iota|}|\underline L L^{m_0+1}\bar\p^{\iota}\O^c\phi(t,x)|
		\lesssim\sum_{\tiny\begin{array}{c}
			|\iota_1|\leq|\iota|,\\ c_1\leq c\end{array}}\dl^{-\kappa+|\iota_1|}|L^{m_0+1}\bar\p^{\iota_1}\O^{c_1}\phi(t,x)|+\dl^{-\kappa},\ (t,x)\in D^1.
		\end{equation*}
		Taking sums over the indices, we get
		\begin{equation}\label{uLL}
		\aligned
		&\sum_{\tiny\begin{array}{c}
			|\iota_1|\leq|\iota|,\\ c_1\leq c\end{array}}\delta^{|\iota_1|}|\underline L L^{m_0+1}\bar\p^{\iota_1}\O^{c_1}\phi(t,x)|
		\\
		\lesssim
		&\sum_{\tiny\begin{array}{c}
			|\iota_1|\leq|\iota|,\\ c_1\leq c\end{array}}\dl^{-\kappa+|\iota_1|}|L^{m_0+1}\bar\p^{\iota_1}\O^{c_1}\phi(t,x)|+\dl^{-\kappa},\,\,\,\,\quad (t,x)\in D^1.
		\endaligned
		\end{equation}	
		Integrating \eqref{uLL} along integral curve of $\underline L$, using the fact $\phi$ vanishes on the outermost characteristic surface, and applying Gronwall inequality, we get
		\begin{equation*}
		|L^{m_0+1}\bar\p^{\iota}\O^c\phi|\lesssim\dl^{1-\kappa-|\iota|}\ \ \text{in}\ D^1.
		\end{equation*}

		\textbf{Step 2: Proof of the rest estimates.}	
		\begin{figure}[htbp]
			\centering
			\includegraphics[scale=0.2]{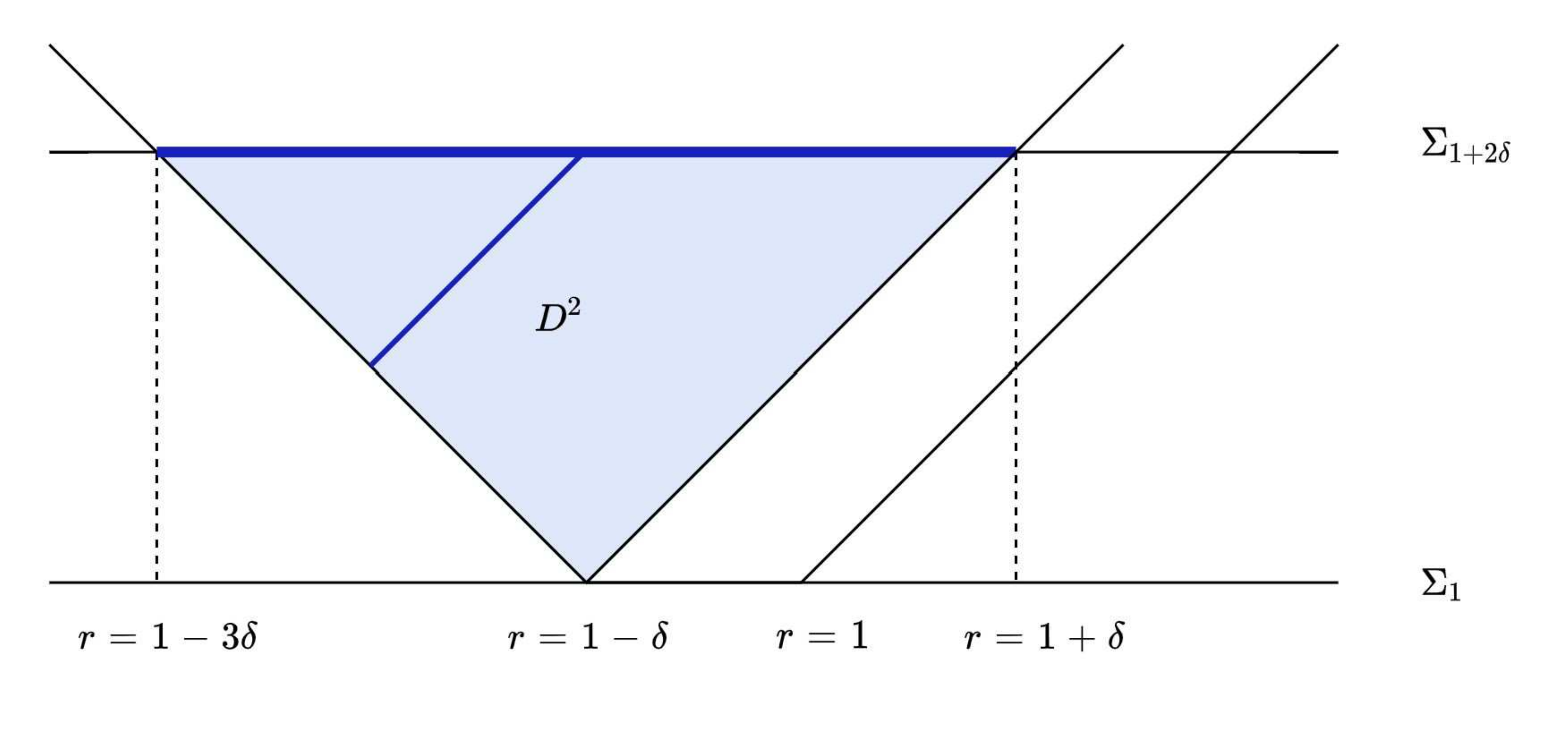}
			\caption{Space domain for $D^2=\{(t,r): 1\le t\le t_0, 2-\dl-t\le r\le t-\dl\}$}\label{pic:p2}
		\end{figure}
		In a similar way, it follows from \eqref{Lle} and \eqref{me} that $|L\underline L\bar\p^\iota\O^c\phi|\lesssim\dl^{-\kappa-|\iota|}$ in $D^2$, integrating
		along integral curves of $L$ yields that for $r\in [1-3\delta,1+\delta]$ (see Figure \ref{pic:p2}),
		\begin{equation*}
		|\underline L\bar\p^\iota\O^c\phi(t_0,x)|\lesssim\delta^{1-|\iota|-\kappa},\quad |\iota|+c+2\leq N_0.
		\end{equation*}
		An induction argument gives that
		\begin{equation}\label{ii}
		\underline L^m\bar\p^\iota\O^c\phi(t_0,x)|\lesssim\delta^{1-|\iota|-\kappa},\quad |\iota|+c+1+m\leq N_0.	
		\end{equation}
		Furthermore, since
		$\p_t=\f12(L+\underline L)$ and $\p_i=\f{\o^i}{2}(L-\underline L)+\f{\o^i_{\perp}}{r}\O$,
		so \eqref{ii}, \eqref{me} and \eqref{Lle} imply that when $2|\al|+c\leq N_0$
		and $r\in [1-3\delta, 1+\delta]$,
		\begin{equation}\label{iid}
		|\p^\al\O^c\phi(t_0,x)|\lesssim
		\left\{
		\begin{aligned}
		\delta^{1-\kappa},\quad&\text{as}\quad|\al|\leq 1,\\
		\delta^{2-|\al|-\kappa},\quad&\text{as}\quad|\al|> 1.
		\end{aligned}
		\right.
		\end{equation}
		{In addition, it follows from $|L\O^c\phi|\lesssim\delta^{1-\kappa}$ in $D^2$ that 
		\begin{equation}\label{Ophi}
		|\O^c\phi(t_0,x)|\lesssim\dl^{2-\kappa},
		\end{equation}}
		which implies \eqref{local3-2} immediately.
		
		Now we are ready to prove the final estimate \eqref{local3-3}. We are already know that
		\begin{equation}\label{LuL}
		|L^m\O^c\phi(t_0,x)|+|\underline L^s\O^c\phi(t_0,x)|\lesssim\dl^{1-\kappa}
		\end{equation}
		from \eqref{Lm} and \eqref{ii} when $|x|\in[1-2\dl,1+\dl]$, $m+c+1\leq N_0$ and $s+c+1\leq N_0$. Assume that
		\begin{equation}\label{Ls}
		|\underline{L}^{s}L\O^c\phi(t_0,x)|\lesssim\dl^{1-\kappa},\quad s+c+2\leq N_0\quad\text{and}\quad0\leq s\leq s_0,
		\end{equation}
		when $|x|\in[1-2\dl,1+\dl]$. It follows from \eqref{me}, \eqref{ii} and \eqref{Ls} that
		\begin{equation*}
		\aligned
		&|\underline L^{s_0+1}L\O^c\phi^I|
		\\
		=&|\underline L^{s_0}\O^c\big(\f{1}{2r}L\phi^I-\f{1}{2r}\underline L\phi^I+\f{1}{r^2}\Omega^2\phi^I-Q^I(\p\phi,\p\phi)\big)|\lesssim\dl^{1-\kappa}
		\endaligned
		\end{equation*}
		for $s_0+c+3\leq N_0$. Thus, \eqref{Ls} holds for any integer $s$ satisfying $s+c+2\leq N_0$. Finally, with an induction argument on $m$, \eqref{me} and \eqref{LuL} yield
		\begin{equation}\label{uLs}
		|\underline L^sL^m\O^c\phi(t_0,x)|\lesssim\dl^{1-\kappa},\quad|x|\in[1-2\dl,1+\dl]\quad\text{and}\quad s+m+c+1\leq N_0.
		\end{equation}
		
		The proof is completed.
	\end{proof}
	
	\section{Global existence near the outermost outgoing cone}\label{YY}
	
	In Section \ref{LE}, we have obtained the local solution for $t\in[1,t_0]$. Moreover, on the hypersurface $\Sigma_{t_0}$, the solution to \eqref{semi} with \eqref{initial} presents different properties in different domains. To be more precise, we see from \eqref{local1-2} that $|L^m \Omega^c \phi|$ stays small in the region $\{ t_0 -r \leq 4\delta \}$ for all $m\geq 0, c\geq 0$.
	This motivates us to distinguish the regions $A_{4\delta}=\{(t,x):t\geq t_0,0\leq t-r\leq 4\delta\}$ and the one inside it.
	Thus, we next divide our proof of global existence into two parts: near the outermost outgoing cone and exactly inside it. In this section, our purpose is to prove the solution to \eqref{semi} equipped with \eqref{local1-2} exists in $A_{4\delta}$ which is near the outermost outgoing cone.
	Worth to mention, the smallness of the width of $A_{4\delta}$ (which is $4\delta$) also plays a vital role in the analysis.

	In this section, we raise and lower indices with the Minkowski metric $m=(m_{\al\beta})=(m^{\al\beta})=\text{diag}(-1,1,1)$ and tacitly sum over repeated indices.
	
	\begin{lemma}\label{nullframe}
		$\{L, {\underline L}, \O\}$ constitutes a null frame with respect to the metric $(m_{\al\beta})$, and admits the following identities:
		\begin{align*}
		&m(L,L)=m({\underline L}, {\underline L})=m(L, \O)=m({\underline L}, \O)=0,\\
		&m(L, {\underline L})=-2,\quad m(\O, \O)=r^2.
		\end{align*}
	\end{lemma}
	\begin{proof}
		Direct calculations yield the  above desired results. We omit the details here.
	\end{proof}
	Inspired by \cite{Ding6} or \cite{Ding3}, one can perform the change of coordinates:
	$(t, x^1, x^2)\longrightarrow (s, u, \theta)$ near $C_0=\{(t, x): t\ge 1+2\dl, t=r\}$ with
	\begin{equation}\label{H0-7}
	\left\{
	\begin{aligned}
	&s=t,\\
	&u=\f12(t-|x|),
	\end{aligned}
	\right.
	\end{equation}
	and $\th$ is the coordinate on the standard circle $\mathbb S^1$. Then, under new coordinate system $(s,u,\theta)$,
	\begin{equation}\label{suth}
	\p_s=L = \p_t+\p_r,\qquad \p_u=\underline L-L= -2\p_r,\qquad \p_\th=\O,
	\end{equation}
	and
	\begin{equation}\label{m}
	m=-4dsdu+4dudu+r^2d\th d\th
	\end{equation}
	by Lemma \ref{nullframe}. Furthermore, we introduce the following subsets (see Figure \ref{pic:03} below):
	\begin{figure}[htbp]
		\centering\includegraphics[scale=0.48]{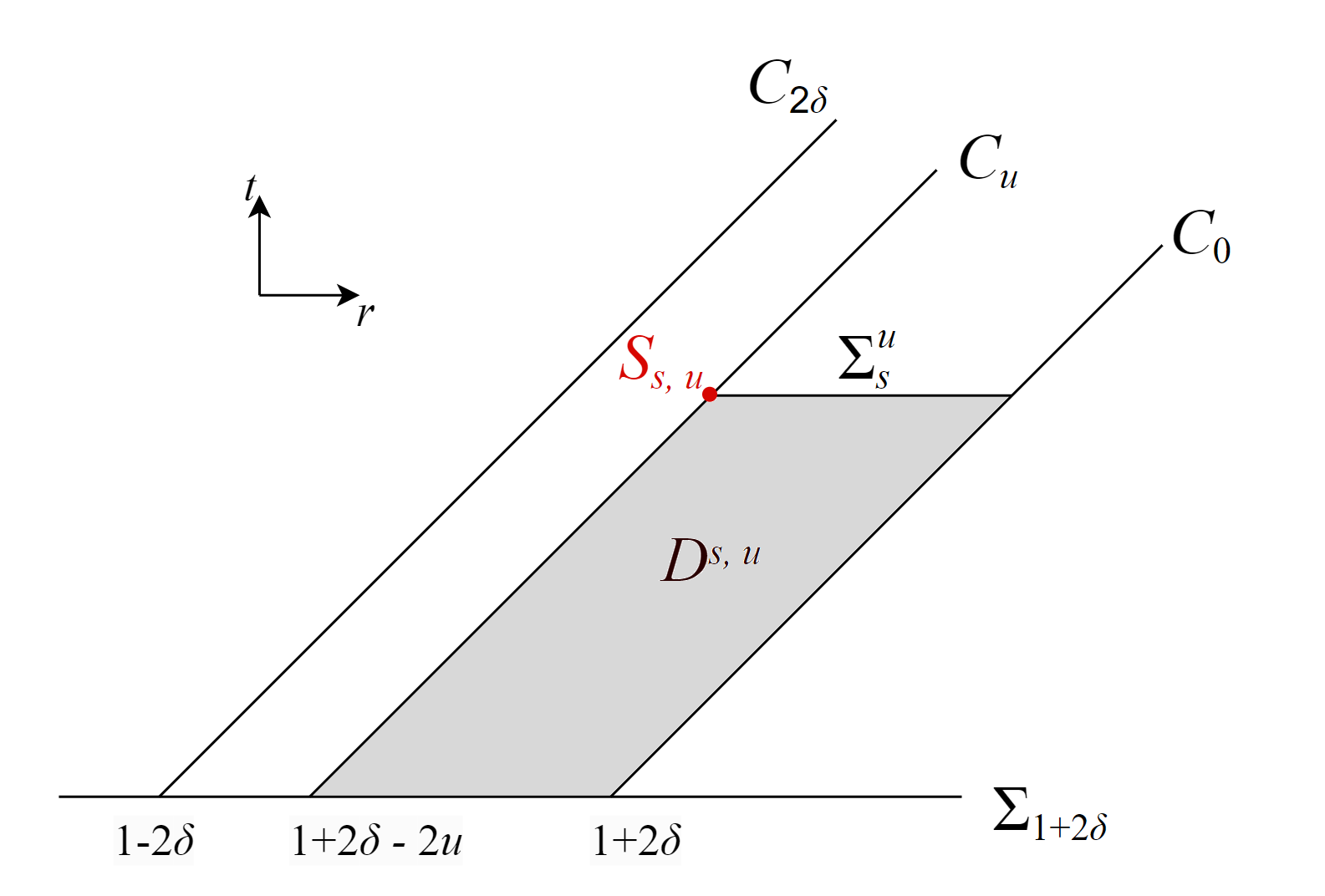}\label{pic:03}
		\caption{ Regional diagram}
	\end{figure}
	
	\begin{definition} Set
		\begin{align*}
		&\Sigma_{s}^{u}:=\{({s}',u',\theta): {s}'=s,0\leq u'\leq  u\},\quad u\in [0, 2\delta],\\
		&C_{u}:=\{({s}',u',\theta): {s}'\geq t_0, u'=u\},\\
		&C_{u}^{s}:=\{({s}',u',\theta): t_0\leq {s}'\leq {s}, u'=u\},\\
		&S_{{s}, u}:=\Sigma_{s}\cap C_{u},\\
		&D^{{s}, u}:=\{({s}', u',\theta): t_0\leq {s}'<{s}, 0\leq u'\leq u\}.
		\end{align*}
	\end{definition}
	
	Next, we introduce some notation for  related integrations.
	\begin{definition}\label{2.3}
		For any continuous function $f$ and tensor field $\xi$, define
		\begin{align*}
		&\int_{ C^s_{u}}f:=\int_{t_0}^s\int_{\mathbb S^1}f(s', u,\theta) d\th ds',\\
		&\int_{\Sigma_s^{u}}f:=\int_0^{u}\int_{\mathbb S^1}f(s, u',\theta) d\th d u',\\
		&\int_{D^{s, u}}f:=\int_{t_0}^s\int_0^{ u}\int_{\mathbb S^1}f(s', u',\theta)d\th d u'ds',\\
		&\|\xi\|_{s,u}:=\sqrt{\int_{\Sigma_s^{ u}}|\xi|^2}.
		\end{align*}
	\end{definition}
	
	For any vector field $J=J^\al\p_\al$, we rewrite it as
	$$J=J^s\p_s+J^u\p_u+\slashed J^\th\p_\th,$$
	where $J^s$, $J^u$ and $\slashed J^\th$ are functions. Let $\mathscr D$ denote the Levi-Civita connection of $m$, then
	\begin{lemma}\label{4.2}
		For any smooth vector field $J$ which vanishes on $C_0$, we have the following divergence identity on $D^{s,u}$:
		\begin{equation}\label{divergence}
		\begin{split}
		-\int_{D^{s,u}}r\mathscr D_\al J^\al =&\f12\int_{\Sigma_s^u}r(J_{\underline L}+J_L)-\f12\int_{\Sigma_{t_0}^u}r(J_{\underline L}+J_L)+\f12\int_{C_u^s}rJ_L,
		\end{split}
		\end{equation}
		where $J_{\underline L}$ and $J_L$ are contractions of $J$ with respect to $\underline L$ and $L$ respectively, that is, $J_{\underline L}=m_{\al\beta}J^\al \underline L^\beta$ and $J_{L}=m_{\al\beta}J^\al L^\beta$.
	\end{lemma}
	
	\begin{proof}
		We can write the divergence $\mathscr D_\al J^\al$ of $J$ under new system $(s,u,\th)$, i.e.,
		\begin{equation}\label{div}
		\mathscr D_\al J^\al=\f{1}{\sqrt{|\det m|}}\big(\p_s(\sqrt{|\det m|}J^s)+\p_u(\sqrt{|\det m|}J^u)+\p_\th(\sqrt{|\det m|}\slashed J^\th)\big),
		\end{equation}
		here $\det m=-4r^2$ by \eqref{m}. In addition, Lemma \ref{nullframe} and \eqref{suth} give that
		\begin{equation}\label{J}
		J^s=-\f12 J_{\underline L}-\f12 J_L,\quad J^u=-\f12J_L\quad \text{and}\ \slashed J^\th=r^{-2}\slashed J_\O
		\end{equation}
		with $\slashed J_\O=m_{\al\beta}J^\al \O^\beta$. Thus, \eqref{divergence} follows from substituting \eqref{J} into \eqref{div} and then integrating $\sqrt{|\det m|}\mathscr D_\al J^\al$ over $D^{s,u}$.
	\end{proof}
	
	Next, we turn to choose suitable vector fields and construct induced energies and fluxed by \eqref{divergence}. Since $\phi$ and its derivatives satisfy the linear equation of the form
	\begin{equation}
	\Box\Psi=\Phi,
	\end{equation}
	where $\Psi$ vanishes on $C_0$, we now focus on deriving the corresponding energies and fluxes with respect to $\Psi$.
	
	\subsection{Preliminary energy estimates}
	The vector fields we will choose are related to the energy-momentum tensor of $\Psi$ which is defined as follows
	\begin{equation}
	Q_{\al\beta}[\Psi]:=(\p_\al\Psi)(\p_\beta\Psi)-\f12m_{\al\beta}\big(-(L\Psi)(\underline L\Psi)+\f{1}{r^2}(\O\Psi)^2\big),
	\end{equation}
	and the vector fields are
	\begin{align}
	J_1&=-\underline u^{2\mu}(m^{\al\beta}Q_{\al\gamma}[\Psi]  L^\gamma\p_\beta),\label{J1}\\
	J_2&=-m^{\al\beta}Q_{\al\gamma}[\Psi]\underline L^\gamma\p_\beta,\label{J2}\\
	J_3&=\f14\underline u^{2\mu-1}\Psi(-\underline L\Psi L-L\Psi\underline L+\f2{r^2}\O\Psi\O)+\f{2\mu-1}{8}\underline u^{2\mu-2}\Psi^2\underline L,\label{J3}
	\end{align}
	where $\mu\in(\f12,1)$ is a fixed constant. 
	\begin{remark}
		$J_1$ and $J_2$ (without $J_3$) are common vector fields which also appeared in \cite{MPY} etc. Different from \eqref{J1}, the authors in \cite{MPY} chose $\mu\in(\f14,\f12)$ since the decay rate of the solution to 3D wave equation is rapid enough. In order to make up for the lack of decay rate in this paper, we are forced to set $\mu\in(\f12,1)$, which however makes it impossible to close the energies due to the present of $(\mu-\f{\underline u}{2r})\f{\underline u^{2\mu-1}}{r^2}(\O\Psi)^2$ in \eqref{DJ12}. To deal with $(\mu-\f{\underline u}{2r})\f{\underline u^{2\mu-1}}{r^2}(\O\Psi)^2$ and get enough decay rate at the same time, we introduce $J_3$, a much more complicated vector field than $J_1$ and $J_2$, which enables us to close the estimates via delicate analysis.
	\end{remark}
	Taking \eqref{J1}-\eqref{J2} to \eqref{divergence} yields
	\begin{align}
	\int_{D^{s,u}}r\mathscr D_\al J_1^\al=&\f12\int_{\Sigma_{s}^u}\underline u^{2\mu}r\big((L\Psi)^2+\f1{r^2}(\O\Psi)^2\big)-\f12\int_{\Sigma_{t_0}^u}\underline u^{2\mu}r\big((L\Psi)^2+\f1{r^2}(\O\Psi)^2\big)\no\\
	&+\f12\int_{C_u^s}r\underline u^{2\mu}(L\Psi)^2,\label{DJ1}\\
	\int_{D^{s,u}}r\mathscr D_\al J_2^\al=&\f12 \int_{\Sigma_{s}^u}r\big((\underline L\Psi)^2+\f1{r^2}(\O\Psi)^2\big)-\f12 \int_{\Sigma_{t_0}^u}r\big((\underline L\Psi)^2+\f1{r^2}(\O\Psi)^2\big)\no\\
	&+\f12\int_{C_u^s}\f1{r}(\O\Psi)^2.\label{DJ2}
	\end{align}
	We note \eqref{divergence} and \eqref{J3} yield the following lemma.
	\begin{lemma}
		For any smooth $\Psi$ vanishing on $C_0$, it holds that
		\begin{equation}\label{DJ3}
		\begin{split}
		\int_{D^{s,u}}r\mathscr D_\al J_3^\al=&\int_{\Sigma_s^u}r\big(-\f12\underline u^{2\mu-1}\Psi L\Psi-\f1{4r}\underline u^{2\mu-1}\Psi^2\big)\\
		&-\int_{\Sigma_{t_0}^u}r\big(-\f12\underline u^{2\mu-1}\Psi L\Psi-\f1{4r}\underline u^{2\mu-1}\Psi^2\big)\\
		&+\int_{C_u^s}r\big(-\f12\underline u^{2\mu-1}\Psi L\Psi-\f1{8r}\underline u^{2\mu-1}\Psi^2\big).
		\end{split}
		\end{equation}
	\end{lemma}
	\begin{proof}
		It follows from Lemmas \ref{nullframe}, \ref{4.2} and \eqref{J3} that
		\begin{equation}\label{divJ3}
		\begin{split}
		&\int_{D^{s,u}}r\mathscr D_\al J_3^\al\\
		=&\int_{\Sigma_s^u}r\big(-\f12\underline u^{2\mu-1}\Psi L\Psi+\f{2\mu-1}8\underline u^{2\mu-2}\Psi^2-\f14\underline u^{2\mu-1}\Psi\p_u\Psi\big)\\
		&+\int_{\Sigma_{t_0}^u}r\big(-\f12\underline u^{2\mu-1}\Psi L\Psi+\f{2\mu-1}8\underline u^{2\mu-2}\Psi^2-\f14\underline u^{2\mu-1}\Psi\p_u\Psi\big)\\
		&+\int_{C^s_{u}}r\big(-\f14\underline u^{2\mu-1}\Psi L\Psi+\f{2\mu-1}8\underline u^{2\mu-2}\Psi^2\big),
		\end{split}
		\end{equation}
		where $\p_u=\underline L-L$. Since $\underline u=t-u$, $r=t-2u$ and $\Psi$ vanished when $u=0$, then
		\begin{equation}\label{rupsi}
		\begin{split}
		&\int_{\Sigma_s^u}-\f r4\underline u^{2\mu-1}\Psi\p_u\Psi\\
		=&-\f18\int_0^u\int_{\mathbb S^1}\big(\p_u(r\underline u^{2\mu-1}\Psi^2)+(2\mu-1)r\underline u^{2\mu-2}\Psi^2+2\underline u^{2\mu-1}\Psi^2\big)d\th du'\\
		=&-\f18\int_{\mathbb S^1}r\underline u^{2\mu-1}\Psi^2d\th-\f18\int_{\Sigma_s^u}r\underline u^{2\mu-2}\Psi^2(2\mu-1+\f{2\underline u}r),
		\end{split}
		\end{equation}
		and hence,
		\begin{equation}\label{t0}
		\int_{\Sigma_{t_0}^u}-\f {r}4\underline u^{2\mu-1}\Psi\p_u\Psi
		=-\f18\int_{\mathbb S^1}r\underline u^{2\mu-1}\Psi^2d\th-\f18\int_{\Sigma_{t_0}^u}r\underline u^{2\mu-2}\Psi^2(2\mu-1+\f{2\underline u}{r}).
		\end{equation}
		It follows from Newton-Leibnitz formula that
		\begin{equation*}
		\begin{split}
		&\int_{\mathbb S^1}r\underline u^{2\mu-1}\Psi^2d\th\\
		=&\int_{\mathbb S^1}(r\underline u^{2\mu-1}\Psi^2)(t_0,u,\th)d\th+\int_{t_0}^s\big(\f{\p}{\p s'}\int_{\mathbb S^1}(s'-2u)(s'-u)^{2\mu-1}\Psi^2(s',u,\th)d\th\big)ds'\\
		=&\int_{\mathbb S^1}(r\underline u^{2\mu-1}\Psi^2)(t_0,u,\th)d\th+\int_{C_u^s}r\big((2\mu-1+\f{\underline u}r)\underline u^{2\mu-2}\Psi^2+2\underline u^{2\mu-1}\Psi L\Psi\big),
		\end{split}
		\end{equation*}
		which implies
		\begin{equation}\label{ru}
		\begin{split}
		&\int_{\Sigma_s^u}-\f r4\underline u^{2\mu-1}\Psi\p_u\Psi\\
		=&-\f18\int_{\mathbb S^1}(r\underline u^{2\mu-1}\Psi^2)(t_0,u,\th)d\th-\f18\int_{\Sigma_s^u}r\underline u^{2\mu-2}\Psi^2(2\mu-1+\f{2\underline u}r)\\
		&-\f18\int_{C_u^s}r\big((2\mu-1+\f{\underline u}r)\underline u^{2\mu-2}\Psi^2+2\underline u^{2\mu-1}\Psi L\Psi\big)
		\end{split}
		\end{equation}
		by \eqref{rupsi}. Substituting  \eqref{ru} and \eqref{t0} into \eqref{divJ3} yields \eqref{DJ3}.
	\end{proof}
	
	It follows from \eqref{DJ1} and \eqref{DJ3} that
	\begin{equation}\label{D13}
	\begin{split}
	&\int_{D^{s,u}}r\big(\mathscr D_\al J_1^\al-\mathscr D_\al J_3^\al\big)\\
	\gtrsim&\int_{\Sigma_s^u}r\Big(\underline u^{2\mu}(L\Psi)^2+\underline u^{2\mu-2}\Psi^2+\underline u^{2\mu}\f{(\O\Psi)^2}{r^2}\Big)+\int_{C_u^s}r(\underline u^{\mu}L\Psi+\f12\underline u^{\mu-1}\Psi)^2\\
	&-\int_{\Sigma_{t_0}^u}r\Big(\underline u^{2\mu}(L\Psi)^2+\underline u^{2\mu-2}\Psi^2+\underline u^{2\mu}\f{(\O\Psi)^2}{r^2}\Big).
	\end{split}
	\end{equation}
	Therefore, it is natural to define the following energies and fluxes
	\begin{align}
	E_1[\Psi](s,u)=&\int_{\Sigma_s^u}r\Big(\underline u^{2\mu}(L\Psi)^2+\underline u^{2\mu-2}\Psi^2+\underline u^{2\mu}\f{(\O\Psi)^2}{r^2}\Big),\label{E1}\\
	E_2[\Psi](s,u)=&\int_{\Sigma_{s}^u}r\big((\underline L\Psi)^2+\f1{r^2}(\O\Psi)^2\big),\label{E2}\\
	F_1[\Psi](s,u)=&\int_{C_u^s}r(\underline u^{\mu}L\Psi+\f12\underline u^{\mu-1}\Psi)^2,\label{F1}\\
	F_2[\Psi](s,u)=&\int_{C_u^s}\f1{r}(\O\Psi)^2,\label{F2}
	\end{align}
	and \eqref{DJ2} and \eqref{D13} are rewritten as
	\begin{align}
	E_1[\Psi](s,u)+F_1[\Psi](s,u)\lesssim& E_1[\Psi](t_0,u)+\int_{D^{s,u}}r\big(\mathscr D_\al J_1^\al-\mathscr D_\al J_3^\al\big),\label{E1F1}\\
	E_2[\Psi](s,u)+F_2[\Psi](s,u)=&E_2[\Psi](t_0,u)+2\int_{D^{s,u}}r\mathscr D_\al J_2^\al.\label{E2F2}
	\end{align}

	In order to close the above energies, it is natural to estimate the second terms on the right hand sides of \eqref{E1F1} and \eqref{E2F2} respectively, which deduces eventually the following theorem.
	
	\begin{theorem}
		For any smooth function $\Psi$ vanishing on $C_0$,
		\begin{equation}\label{EF}
		\begin{split}
		&E_1[\Psi](s,u)+F_1[\Psi](s,u)+\dl E_2[\Psi](s,u)+\dl F_2[\Psi](s,u)\\
		\lesssim&E_1[\Psi](t_0,u)+\dl E_2[\Psi](t_0,u)+|\int_{D^{s,u}}r\Phi(\underline u^{2\mu}L\Psi+\f12\underline u^{2\mu-1}\Psi)|+\dl|\int_{D^{s,u}}r\Phi\underline L\Psi|.
		\end{split}
		\end{equation}
	\end{theorem}
	\begin{proof}
		For any vector field $V$,
		\begin{equation}\label{DJ}
		-\mathscr D_\al(m^{\al\beta}Q_{\beta\g}V^\g)=-\Phi(V\Psi)-\f12m^{\al\al'}m^{\beta\beta'}Q_{\al'\beta'}[\Psi]\leftidx^{{(V)}}{\pi}_{\al\beta},
		\end{equation}
		where $\leftidx^{{(V)}}{\pi}_{\al\beta}:=m(\mathscr{D}_{\al}V,\p_{\beta})+m(\mathscr{D}_{\beta}V,\p_{\al})$ is a deformation tensor. Since $m^{\al\beta}=-\f12L^{\al}\underline L^\beta-\f12L^{\beta}\underline L^\al+\f1{r^2}\O^\al\O^\beta$, then
		\begin{equation}\label{mmQ}
		\begin{split}
		m^{\al\al'}m^{\beta\beta'}Q_{\al'\beta'}[\Psi]\leftidx^{{(V)}}{\pi}_{\al\beta}=&\f14Q_{\underline L\underline L}[\Psi]\leftidx^{{(V)}}{\pi}_{LL}+\f12Q_{L\underline L}[\Psi]\leftidx^{{(V)}}{\pi}_{L\underline L}+\f14Q_{LL}[\Psi]\leftidx^{{(V)}}{\pi}_{\underline L\underline L}\\
		&-\f1{r^2}Q_{\underline L\O}[\Psi]\leftidx^{{(V)}}{\pi}_{L\O}-\f1{r^2}Q_{L\O}[\Psi]\leftidx^{{(V)}}{\pi}_{\underline L\O}+\f1{r^4}Q_{\O\O}[\Psi]\leftidx^{{(V)}}{\pi}_{\O\O}
		\end{split}
		\end{equation}
		with $Q_{XY}[\Psi]=Q_{\al\beta}[\Psi]X^\al Y^\beta$ and $\leftidx^{{(V)}}{\pi}_{XY}=\leftidx^{{(V)}}{\pi}_{\al\beta}X^\al Y^\beta$ for any vectorfields $X$ and $Y$. Since
		\begin{align*}
		&\leftidx^{{(\underline u^{2\mu} L)}}{\pi}_{LL}=\leftidx^{{(\underline u^{2\mu} L)}}{\pi}_{\underline L\underline L}=\leftidx^{{(\underline u^{2\mu} L)}}{\pi}_{L\O}=\leftidx^{{(\underline u^{2\mu} L)}}{\pi}_{\underline L\O}=0,\\
		&\leftidx^{{(\underline u^{2\mu} L)}}{\pi}_{L\underline L}=-4\mu\underline u^{2\mu-1},\quad \leftidx^{{(\underline u^{2\mu} L)}}{\pi}_{\O\O}=2r\underline u^{2\mu},\\
		&\leftidx^{{(\underline L)}}{\pi}_{LL}=\leftidx^{{(\underline L)}}{\pi}_{\underline L\underline L}=\leftidx^{{(\underline L)}}{\pi}_{L\underline L}=\leftidx^{{(\underline L)}}{\pi}_{ L\O}=\leftidx^{{(\underline L)}}{\pi}_{\underline L\O}=0,\quad \leftidx^{{(\underline L)}}{\pi}_{\O\O}=-2r,
		\end{align*}
		and
		\begin{align*}
		&Q_{LL}[\Psi]=(L\Psi)^2,\quad Q_{\underline L\underline L}[\Psi]=(\underline L\Psi)^2,\quad Q_{L\underline L}[\Psi]=\f1{r^2}(\O\Psi)^2,\quad Q_{L\O}[\Psi]=(L\Psi)(\O\Psi),\\
		&Q_{\underline L\O}[\Psi]=(\underline L\Psi)(\O\Psi),\quad Q_{\O\O}[\Psi]=\f12(\O\Psi)^2+\f{r^2}2(L\Psi)(\underline L\Psi),
		\end{align*}
		then it follows from \eqref{DJ} and \eqref{mmQ} that
		\begin{align}
		&\mathscr D_\al J_1^\al=-\Phi(\underline u^{2\mu}L\Psi)+(\mu-\f{\underline u}{2r})\f{\underline u^{2\mu-1}}{r^2}(\O\Psi)^2-\f{1}{2r}\underline u^{2\mu}(L\Psi)(\underline L\Psi),\label{DJ12}\\
		&\mathscr D_\al J_2^\al=-\Phi(\underline L\Psi)+\f{1}{2r^3}(\O\Psi)^2+\f{1}{2r}(L\Psi)(\underline L\Psi),\label{DJ22}
		\end{align}
		which implies
		\begin{equation*}
		\begin{split}
		|\int_{D^{s,u}}r\mathscr D_\al J_2^\al|\lesssim|&\int_{D^{s,u}}r\Phi(\underline L\Psi)|+\int_0^u F_2[\Psi](s,u')du'\\
		&+\int_{t_0}^s\tau^{-1-\mu}\big(E_1[\Psi](\tau,u)+E_2[\Psi](\tau,u)\big)d\tau.
		\end{split}
		\end{equation*}
		This, together with \eqref{E2F2} and Gronwall's inequality yields
		\begin{equation}\label{EF2}
		E_2[\Psi](s,u)+F_2[\Psi](s,u)\lesssim|\int_{D^{s,u}}r\Phi(\underline L\Psi)|+ E_2[\Psi](t_0,u)+\int_{t_0}^s\tau^{-1-\mu}E_1[\Psi](\tau,u)d\tau.
		\end{equation}
		In addition, since
		\begin{equation*}
		\mathscr D_\al J_3^\al=\f12\underline u^{2\mu-1}\big(-L\Psi\underline L\Psi+\f1{r^2}(\O\Psi)^2\big)+\f12\underline u^{2\mu-1}\Psi\Phi-\f{2\mu-1}{8r}\underline u^{2\mu-2}\Psi^2
		\end{equation*}
		by \eqref{J3}, then together with \eqref{DJ12} yields
		\begin{equation*}
		\begin{split}
		\mathscr D_\al J_1^\al-\mathscr D_\al J_3^\al=&-\Phi(\underline u^{2\mu}L\Psi+\f12\underline u^{2\mu-1}\Psi)-\f{u}{2r}\underline u^{2\mu-1}L\Psi\underline L\Psi\\
		&+\f{2\mu-1}{8r}\underline u^{2\mu-2}\Psi^2+(\mu-1-\f u{2r})\underline u^{2\mu-1}\f1{r^2}(\O\Psi)^2\\
		\lesssim&-\Phi(\underline u^{2\mu}L\Psi+\f12\underline u^{2\mu-1}\Psi)-\f{u}{2r}\underline u^{2\mu-1}L\Psi\underline L\Psi+\f{2\mu-1}{8r}\underline u^{2\mu-2}\Psi^2.
		\end{split}
		\end{equation*}
		Hence, we have
		\begin{equation}\label{rDJ13}
		\begin{split}
		&\int_{D^{s,u}}r\big(\mathscr D_\al J_1^\al-\mathscr D_\al J_3^\al\big)\\
		\lesssim&|\int_{D^{s,u}}r\Phi(\underline u^{2\mu}L\Psi+\f12\underline u^{2\mu-1}\Psi)|+\int_{D^{s,u}}\big(r\underline u^{2\mu-1}|L\Psi|^2+r\dl^2\underline u^{2\mu-3}|\underline L\Psi|^2+r\underline u^{2\mu-3}\Psi^2\big)\\
		\lesssim&|\int_{D^{s,u}}r\Phi(\underline u^{2\mu}L\Psi+\f12\underline u^{2\mu-1}\Psi)|+\int_{D^{s,u}}r(\underline u^{\mu}L\Psi+\f12\underline u^{\mu-1}\Psi)^2\\
		&+\dl^2\int_{D^{s,u}}r\underline u^{2\mu-3}|\underline L\Psi|^2+\int_{D^{s,u}}r\underline u^{2\mu-3}\Psi^2.
		\end{split}
		\end{equation}
		For the last term in \eqref{rDJ13}, it cannot be absorbed by left side hand of \eqref{E1F1} directly, so one can estimate it separately. As $\Psi$ vanishes on $C_0$, then
		\begin{equation*}
		\begin{split}
		\int_{C_u^s}r\underline u^{2\mu-3}\Psi^2=&\int_{D^{s,u}}r\Big((3-2\mu)\underline u^{2\mu-4}\Psi^2-2\underline u^{2\mu-3}\Psi^2+2\underline u^{2\mu-3}\Psi\p_u\Psi\Big)\\
		\lesssim&\dl^{-1}\int_{D^{s,u}}r\underline u^{2\mu-3}\Psi^2+\dl\int_{D^{s,u}}r\underline u^{2\mu-3}\big((\underline L\Psi)^2+(L\Psi)^2\big),
		\end{split}
		\end{equation*}
		which implies
		\begin{equation*}
		\int_{C_u^s}r\underline u^{2\mu-3}\Psi^2\lesssim\dl\int_{D^{s,u}}r\underline u^{2\mu-3}\big((\underline L\Psi)^2+(L\Psi)^2\big),
		\end{equation*}
		and hence,
		\begin{equation}\label{DPsi}
		\int_{D^{s,u}}r\underline u^{2\mu-3}\Psi^2\lesssim\dl^2\int_{D^{s,u}}r\underline u^{2\mu-3}(\underline L\Psi)^2+\dl^2\int_{D^{s,u}}\underline u^{-3}r(\underline u^\mu L\Psi)^2.
		\end{equation}
		Taking \eqref{DPsi} to \eqref{rDJ13} yields
		\begin{equation}\label{r13}
		\begin{split}
		\int_{D^{s,u}}r\big(\mathscr D_\al J_1^\al-\mathscr D_\al J_3^\al\big)\lesssim&|\int_{D^{s,u}}r\Phi(\underline u^{2\mu}L\Psi+\f12\underline u^{2\mu-1}\Psi)|+\int_0^uF_1[\Psi](s,u')du'\\
		&+\dl^2\int_{t_0}^s\tau^{2\mu-3}E_2[\Psi](\tau,u)d\tau+\dl^2\int_{t_0}^s\tau^{-3}E_1[\Psi](\tau,u)d\tau.
		\end{split}
		\end{equation}
		Hence, it follows from \eqref{E1F1} and \eqref{r13} that
		\begin{equation}\label{EF1}
		\begin{split}
		&E_1[\Psi](s,u)+F_1[\Psi](s,u)\\
		\lesssim& E_1[\Psi](t_0,u)+|\int_{D^{s,u}}r\Phi(\underline u^{2\mu}L\Psi+\f12\underline u^{2\mu-1}\Psi)|+\dl^2\int_{t_0}^s\tau^{2\mu-3}E_2[\Psi](\tau,u)d\tau.
		\end{split}
		\end{equation}
		Combine \eqref{EF1} and \eqref{EF2}, we get \eqref{EF}.
	\end{proof}
	
	The remaining task in this section is to estimate $\int_{D^{s,u}}r\Phi(\underline u^{2\mu}L\Psi+\f12\underline u^{2\mu-1}\Psi)$ in \eqref{EF1} and $\dl\int_{D^{s,u}}r\Phi\underline L\Psi$ which appear in the right hand side of \eqref{EF}. In this paper, $\Psi$ is chosen as $\Psi_k^I:=Z^k\phi^I$ and hence $\Phi=\Phi_k^I=\Box\Psi_k^I$, where $Z^k=Z_1\cdots Z_k$ with $Z_j\in\{\p_\al, S, H_i,\O\}$, $1\leq j\leq k$. The energies and fluxes corresponding to $\phi$ are defined as follows:
	\begin{align}
	E_{i,k}(s, u)&=\sum_{I=1}^N\sum_{\bar Z\in\{S, H_i,\O\}}\delta^{2l}E_i[\p^l\bar Z^{k-l}\phi^I](s, u),\quad i=1,2,\label{ei}\\
	F_{i,k}(s, u)&=\sum_{I=1}^N\sum_{\bar Z\in\{S, H_i,\O\}}\delta^{2l}F_i[\p^l\bar Z^{k-l}\phi^I](s, u),\quad i=1,2,\label{fi}\\
	\mathscr E_{k}(s, u)&=\sum_{0\leq n\leq k}\big(E_{1,n}(s, u)+\dl E_{2,n}(s, u)\big),\label{eil}\\
	\mathscr F_{k}(s, u)&=\sum_{0\leq n\leq k}\big(F_{1,n}(s, u)+\dl F_{2,n}(s, u)\big).\label{fil}
	\end{align}

	We make the following bootstrap assumptions:
	\begin{equation}\label{BA}
	\begin{split}
	&\dl^l\|LZ^k\phi\|_{L^\infty(\Sigma_s^u)}+\dl^l\|\f 1 r\O Z^k\phi\|_{L^\infty(\Sigma_s^u)}\lesssim\dl^{1-\kappa}s^{-3/2}M,\\
	&\dl^l\|\underline LZ^k\phi\|_{L^\infty(\Sigma_s^u)}\lesssim\dl^{-\kappa}s^{-1/2}M,
	\end{split}
	\end{equation}
	where $k\leq N_1$, $N_1$ is a fixed large positive integer, $M$ is some positive constant to be suitably chosen later, $Z\in\{\p,S,H_i,\O\}$, and $l$ is the number of $\p$ included in $Z^k$. With this assumption, we then get the following results.
	\begin{theorem}
		Under the assumption \eqref{BA} with $\dl>0$ small enough, it holds that
		\begin{equation}\label{E1F}
		\mathscr E_{2N_1+1}(s,u)+\mathscr F_{2N_1+1}(s,u)\lesssim\dl^{2-2\kappa}s^{2\varsigma},
		\end{equation}
		where $\varsigma$ is some constant multiple of $\dl^{1-2\kappa}M^2$.
		Moreover, we have
		\begin{equation}\label{LOuL}
		\begin{split}
		&\sum_{k\leq 2N_1+1}\dl^l\|LZ^k\phi\|_{L^2(\Sigma_s^u)}+\sum_{k\leq 2N_1+1}\dl^l\|\f1r\O Z^k\phi\|_{L^2(\Sigma_s^u)}\lesssim\dl^{1-\kappa}s^{\varsigma-\mu},\\
		&\sum_{k\leq 2N_1+1}\dl^l\|\underline LZ^k\phi\|_{L^2(\Sigma_s^u)}\lesssim\dl^{\f12-\kappa}s^{\varsigma}.
		\end{split}
		\end{equation}
	\end{theorem}	
	\begin{proof}
		Since $\phi^I$ satisfies \eqref{semi}, it follows from \eqref{cnull} that
		\begin{equation}\label{BZp}
		\begin{split}
		\dl^l|\Box Z^k\phi^I|\lesssim\sum_{k_1+k_2\leq k}\big(\dl^{l_1}|\underline LZ^{k_1}\phi|+\dl^{l_1}|\f1r\O Z^{k_1}\phi|\big)\big(\dl^{l_2}|LZ^{k_2}\phi|+\dl^{l_2}|\f1r\O Z^{k_2}\phi|\big),
		\end{split}
		\end{equation}
		where $l_i$ $(i=1,2)$ is the corresponding number of $\p$ in $Z^{k_i}$.
		
		When $k\leq 2N_1+1$ and $k_1\leq k_2$, then $k_1\leq N_1$, with the help of \eqref{BA} yields
		\begin{equation}\label{1leq2}
		\begin{split}
		&\dl\int_{D^{s,u}}r\big(\dl^{l_1}|\underline LZ^{k_1}\phi|+\dl^{l_1}|\f1r\O Z^{k_1}\phi|\big)\big(\dl^{l_2}|LZ^{k_2}\phi|+\dl^{l_2}|\f1r\O Z^{k_2}\phi|\big)
		\dl^l|\underline LZ^k\phi^I|\\
		\lesssim&\dl^{1-2\kappa}M^2\int_{D^{s,u}}\tau^{-\f12-\mu}r\big(\dl^{2l_2}\underline u^{2\mu}|LZ^{k_2}\phi|^2+\dl^{2l_2}\underline u^{2\mu}|\f1r\O Z^{k_2}\phi|^2\big)\\
		&+
		\dl\int_{D^{s,u}}\tau^{-\f12-\mu}r\dl^{2l}|\underline LZ^k\phi^I|^2\\
		\lesssim&\dl^{1-2\kappa}M^2\int_{t_0}^s\tau^{-\f12-\mu}E_{1,\leq 2N_1+1}(\tau,u)d\tau+\dl\int_{t_0}^s\tau^{-\f12-\mu}E_{2,\leq 2N_1+1}(\tau,u)d\tau
		\end{split}
		\end{equation}
		and
		\begin{equation}
		\begin{split}
		&\int_{D^{s,u}}r\big(\dl^{l_1}|\underline LZ^{k_1}\phi|+\dl^{l_1}|\f1r\O Z^{k_1}\phi|\big)\big(\dl^{l_2}|LZ^{k_2}\phi|+\dl^{l_2}|\f1r\O Z^{k_2}\phi|\big)\dl^l|\underline u^{2\mu}LZ^k\phi^I+\f12\underline u^{2\mu-1}Z^k\phi^I|\\
		\lesssim&\dl^{-1}\int_{D^{s,u}}r\dl^{2l}|\underline u^{\mu}LZ^k\phi^I+\f12\underline u^{\mu-1}Z^k\phi^I|^2\\
		&+\dl^{1-2\kappa}M^2\int_{D^{s,u}}\tau^{-1}r\big(\dl^{2l_2}\underline u^{2\mu}|LZ^{k_2}\phi|^2+\dl^{2l_2}\underline u^{2\mu}|\f1r\O Z^{k_2}\phi|^2\big)\\
		\lesssim&\dl^{-1}\int_0^uF_{1,\leq 2N_1+1}(s,u')du'+\dl^{1-2\kappa}M^2\int_{t_0}^s\tau^{-1}E_{1,\leq 2N_1+1}(\tau,u)d\tau.
		\end{split}
		\end{equation}
		
		Similarly, when $k\leq 2N_1+1$ and $k_1> k_2$, one has
		\begin{equation}
		\begin{split}
		&\dl\int_{D^{s,u}}r\big(\dl^{l_1}|\underline LZ^{k_1}\phi|+\dl^{l_1}|\f1r\O Z^{k_1}\phi|\big)\big(\dl^{l_2}|LZ^{k_2}\phi|+\dl^{l_2}|\f1r\O Z^{k_2}\phi|\big)
		\dl^l|\underline LZ^k\phi^I|\\
		\lesssim&\dl^{2-\kappa}M\int_{t_0}^s\tau^{-3/2}E_{2,\leq 2N_1+1}(\tau,u)d\tau
		\end{split}
		\end{equation}
		and
		\begin{equation}\label{1>2}
		\begin{split}
		&\int_{D^{s,u}}r\big(\dl^{l_1}|\underline LZ^{k_1}\phi|+\dl^{l_1}|\f1r\O Z^{k_1}\phi|\big)\big(\dl^{l_2}|LZ^{k_2}\phi|+\dl^{l_2}|\f1r\O Z^{k_2}\phi|\big)\dl^l|\underline u^{2\mu}LZ^k\phi^I+\f12\underline u^{2\mu-1}Z^k\phi^I|\\
		\lesssim&\dl^{-1}\int_0^uF_{1,\leq 2N_1+1}(s,u')du'+\dl^{3-2\kappa}M^2\int_{t_0}^s\tau^{-3+2\mu}E_{2\leq 2N_1+1}(\tau,u)d\tau.
		\end{split}
		\end{equation}
		
		\eqref{E1F} follows from inserting \eqref{1leq2}-\eqref{1>2} into \eqref{EF} and using Gronwall's inequality, which implies \eqref{LOuL} immediately.
	\end{proof}
	
	\begin{remark}\label{rem:weak}
		In this subsection, we are able to control the power of $\dl^{-1}$ to the weighted energies and fluxes (see \eqref{E1F}), but unfortunately an unfavorable growth $s^\varsigma$ occurs in the energies. It is thus not yet possible to close the bootstrap assumption \eqref{BA} since $\|\underline LZ^k\phi\|_{L^2(\Sigma_s^u)}$ increases as time evolves (see \eqref{LOuL}).
	\end{remark}
	
	\subsection{Strong energy estimates}
	
	As we explained in Remark \ref{rem:weak}, \eqref{E1F} is not yet enough for us to close the assumption \eqref{BA} and then obtain the global existence of solution in $A_{4\dl}$. We now aim to get bounds for the energies that do not have any growth in time.
	The strategy is to reduce the power of $\underline u$ in \eqref{J1}, that is, one can introduce another vector field
	\begin{equation}\label{J4}
	J_4=-\underline u^{2\nu}(m^{\al\beta}Q_{\al\g}[\Psi]L^\g\p_\beta)
	\end{equation}
	with $\nu\in(0,\f12)$. Set
	\begin{align*}
	\tilde E_1[\Psi](s,u)=&\int_{\Sigma_s^u}r\underline u^{2\nu}\Big((L\Psi)^2+\f1{r^2}(\O\Psi)^2\Big),\\
	\tilde F_1[\Psi](s,u)=&\int_{C_u^s}r\underline u^{2\nu}(L\Psi)^2,
	\end{align*}
	and
	\begin{align*}
	\t E_{1,k}(s, u)&=\sum_{I=1}^N\sum_{Z}\delta^{2l}\t E_1[Z^{k}\phi^I](s, u),\quad \mathscr {\t E}_{k}(s, u)=\sum_{0\leq n\leq k}\big(\t E_{1,n}(s, u)+\dl E_{2,n}(s, u)\big),\\
	\t F_{1,k}(s, u)&=\sum_{I=1}^N\sum_{Z}\delta^{2l}\t F_1[Z^{k}\phi^I](s, u),\quad \mathscr{\t F}_{k}(s, u)=\sum_{0\leq n\leq k}\big(\t F_{1,n}(s, u)+\dl F_{1,n}(s, u)\big).
	\end{align*}
	Taking \eqref{J4} to \eqref{divergence} yields
	\begin{equation}\label{tE1F1}
	\begin{split}
	&\t E_1[\Psi](s,u)+\t F_1[\Psi](s,u)=\t E_1[\Psi](t_0,u)+2\int_{D^{s,u}}r\mathscr D_\al J_4^\al\\
	=&\t E_1[\Psi](t_0,u)-2\int_{D^{s,u}}r\Phi\underline u^{2\nu}L\Psi-\int_{D^{s,u}}rm^{\al\al'}m^{\beta\beta'}Q_{\al'\beta'}[\Psi]\leftidx^{{(\underline u^{2\nu}L)}}{\pi}_{\al\beta}\\
	=&\t E_1[\Psi](t_0,u)-2\int_{D^{s,u}}r\Phi\underline u^{2\nu}L\Psi+\int_{D^{s,u}}r\big(\underline u^{2\nu-1}(2\nu-\f{\underline u}r)\f{(\O\Psi)^2}{r^2}-\f1r\underline u^{2\nu}(L\Psi)(\underline L\Psi)\big).
	\end{split}
	\end{equation}	
	Since $\nu\in(0,\f12)$ and $\underline ur^{-1}\geq 1$, then one gets
	\begin{equation*}
	\begin{split}
	&\int_{D^{s,u}}r\big(\underline u^{2\nu-1}(2\nu-\f{\underline u}r)\f{(\O\Psi)^2}{r^2}-\f1r\underline u^{2\nu}(L\Psi)(\underline L\Psi)\big)\\
	\lesssim&\dl^{-1}\int_{D^{s,u}}r\underline u^{2\nu}(L\Psi)^2+\dl\int_{D^{s,u}}r^{-1}\underline u^{2\nu}(\underline L\Psi)^2\\
	\lesssim&\dl^{-1}\int_0^u\t F_1[\Psi](s,u')du'+\dl\int_{t_0}^s\tau^{2\nu-2}E_2[\Psi](\tau,u)d\tau,
	\end{split}
	\end{equation*}
	which gives
	\begin{equation}
	\t E_1[\Psi](s,u)+\t F_1[\Psi](s,u)\lesssim\t E_1[\Psi](t_0,u)+|\int_{D^{s,u}}r\Phi\underline u^{2\nu}L\Psi|+\dl\int_{t_0}^s\tau^{2\nu-2}E_2[\Psi](\tau,u)d\tau
	\end{equation}	
	with the help of \eqref{tE1F1}. In addition, the same argument with \eqref{EF2} yields
	\begin{equation}\label{tEF2}
	E_2[\Psi](s,u)+F_2[\Psi](s,u)\lesssim E_2[\Psi](t_0,u)+ |\int_{D^{s,u}}r\Phi(\underline L\Psi)|+ \int_{t_0}^s\tau^{-1-\nu}\t E_1[\Psi](\tau,u)d\tau,
	\end{equation}
	thus,
	\begin{equation}\label{tEF}
	\begin{split}
	&\t E_1[\Psi](s,u)+\t F_1[\Psi](s,u)+\dl E_2[\Psi](s,u)+\dl F_2[\Psi](s,u)\\
	\lesssim&\t E_1[\Psi](t_0,u)+\dl E_2[\Psi](t_0,u)+|\int_{D^{s,u}}r\Phi(\underline u^{2\nu} L\Psi)|+\dl|\int_{D^{s,u}}r\Phi\underline L\Psi|.
	\end{split}
	\end{equation}
	Let $\Psi=Z^k\phi^I$ and $\Phi=\Box Z^k\phi^I$ in \eqref{tEF}, and the following results hold. 
	\begin{theorem}
		Under the assumption \eqref{BA} with $\dl>0$ small enough, it holds that
		\begin{equation}\label{Zkp}
		\dl^l\|Z^k\phi\|_{L^\infty(S_{s,u})}\lesssim\dl^{1-\kappa}s^{-1/2}, \quad k\leq 2N_1.
		\end{equation}
	\end{theorem}
	\begin{proof}
		When $k\leq 2N_1+1$ and $k_1\leq k_2$, then $k_1\leq N_1$, with the help of \eqref{BA} and \eqref{LOuL} yields
		\begin{equation}\label{k1<k2}
		\begin{split}
		&\dl\int_{D^{s,u}}r\big(\dl^{l_1}|\underline LZ^{k_1}\phi|+\dl^{l_1}|\f1r\O Z^{k_1}\phi|\big)\big(\dl^{l_2}|LZ^{k_2}\phi|+\dl^{l_2}|\f1r\O Z^{k_2}\phi|\big)
		\dl^l|\underline LZ^k\phi^I|\\
		\lesssim&\dl^{1-\kappa}M\int_{t_0}^s\tau^{-1/2}\big(\dl^{l_2}\|LZ^{k_2}\phi\|_{L^2(\Sigma_\tau^u)}+\dl^{l_2}\|\f1r\O Z^{k_2}\phi\|_{L^2(\Sigma_\tau^u)}\big)\dl^l\|\underline LZ^k\phi^I\|_{L^2(\Sigma_\tau^u)}\\
		\lesssim&\int_{t_0}^s\dl^{\f52-3\kappa}M\tau^{-\f12-\mu+2\varsigma}d\tau\lesssim\dl^{2-2\kappa}
		\end{split}
		\end{equation}
		and
		\begin{equation}
		\begin{split}
		&\int_{D^{s,u}}r\big(\dl^{l_1}|\underline LZ^{k_1}\phi|+\dl^{l_1}|\f1r\O Z^{k_1}\phi|\big)\big(\dl^{l_2}|LZ^{k_2}\phi|+\dl^{l_2}|\f1r\O Z^{k_2}\phi|\big)\underline u^{2\nu}
		\dl^l|LZ^k\phi^I|\\
		\lesssim&\dl^{-1}\int_0^u\t F_{1,\leq 2N_1+1}(s,u')du'+\int_0^u F_2(s,u')du'.
		\end{split}
		\end{equation}
		
		When $k\leq 2N_1+1$ and $k_1> k_2$, we have
		\begin{equation}
		\begin{split}
		&\dl\int_{D^{s,u}}r\big(\dl^{l_1}|\underline LZ^{k_1}\phi|+\dl^{l_1}|\f1r\O Z^{k_1}\phi|\big)\big(\dl^{l_2}|LZ^{k_2}\phi|+\dl^{l_2}|\f1r\O Z^{k_2}\phi|\big)
		\dl^l|\underline LZ^k\phi^I|\\
		\lesssim&\dl\int_{t_0}^s\tau^{-3/2}E_{2,\leq 2N_1+1}(\tau,u)d\tau
		\end{split}
		\end{equation}
		and
		\begin{equation}\label{k1>k2}
		\begin{split}
		&\int_{D^{s,u}}r\big(\dl^{l_1}|\underline LZ^{k_1}\phi|+\dl^{l_1}|\f1r\O Z^{k_1}\phi|\big)\big(\dl^{l_2}|LZ^{k_2}\phi|+\dl^{l_2}|\f1r\O Z^{k_2}\phi|\big)\underline u^{2\nu}
		\dl^l|LZ^k\phi^I|\\
		\lesssim&\dl^{-1}\int_0^u\t F_{1,\leq 2N_1+1}(s,u')du'+\dl^{3-2\kappa}M^2\int_{t_0}^s\tau^{-3+2\nu}E_{2\leq 2N_1+1}(\tau,u)d\tau.
		\end{split}
		\end{equation}
		Taking \eqref{k1<k2}-\eqref{k1>k2} to \eqref{tEF} yields
		\begin{equation}\label{t}
		\mathscr{\t E}_{2N_1+1}(s,u)+\mathscr{\t F}_{2N_1+1}(s,u)\lesssim\dl^{2-2\kappa}.
		\end{equation}
		
		For any smooth function $f$ vanishing on $C_0$,
		\begin{equation*}
		\begin{split}
		\int_{\mathbb S^1}rf^2d\th=&\int_0^u\int_{\mathbb S^1}\big(2f(\underline Lf-Lf)r-2f^2\big)d\th du'\\
		\lesssim&\dl^{-1}\int_0^u\int_{\mathbb S^1} f^2rd\th du'+\dl\int_{\Sigma_s^u}r\big((\underline Lf)^2+(Lf)^2\big),
		\end{split} 
		\end{equation*}
		that is,
		\begin{equation*}
		\|f\|^2_{L^2(S_{s,u})}=\int_{\mathbb S^1}rf^2d\th\lesssim\dl E_2[f](s,u)+\dl s^{-2\nu}\t E_1[f](s,u).
		\end{equation*}
		Thus, for any $k\leq 2N_1$,
		\begin{equation*}
		\begin{split}
		&\dl^l\|Z^k\phi\|_{L^\infty(S_{s,u})}\lesssim\f{\dl^l}{\sqrt r}\|\O^{\leq 1}Z^k\phi\|_{L^2(S_{s,u})}\\
		\lesssim&s^{-1/2}\dl^{1/2}\big(\sqrt{E_{2,\leq 2N_1+1}(s,u)}+s^{-\nu}\sqrt{\t E_{1,\leq 2N_1+1}(s,u)}\big)\lesssim\dl^{1-\kappa}s^{-1/2}
		\end{split}
		\end{equation*}
		by \eqref{t} and hence \eqref{Zkp} is proved.
	\end{proof}
	
	\subsection{$L^{\infty}$ estimates on $C_\dl$}
	
	Till now, we have closed the bootstrap assumption \eqref{BA} and hence obtained the existence of solution in $A_{4\dl}$. In order to solve the initial boundary value problem in the next section, here we can get improved estimate of $\phi$ on $C_{\dl}$. In fact, it follows from \eqref{me} that
	\begin{equation}\label{LruL}
	L(r^{1/2}\underline L\phi^I)=\f{L\phi^I}{2r^{1/2}}+\f1{r^{3/2}}\O^2\phi^I-r^{1/2}Q^I(\p\phi,\p\phi),
	\end{equation}
	then we have $|L(r^{1/2}\underline L\phi^I)|\lesssim\dl^{1-\kappa}s^{-2}+\dl^{1-\kappa}s^{-3/2}r^{1/2}|\underline L\phi|$ by \eqref{Zkp}. This, together with \eqref{local3-3}, deduces
	\begin{equation*}
	|\underline L\phi|_{C_\dl}\lesssim\dl^{1-\kappa}t^{-1/2}
	\end{equation*} 
	when integrating along integral curves of $L$. We assume that
	\begin{equation}\label{assume}
	|\underline L\bar Z^k\phi|_{C_\dl}\lesssim\dl^{1-\kappa}t^{-1/2},\quad\text{for}\quad\bar Z\in\{S,H_i,\O\}\quad\text{and}\quad k\leq k_0\leq 2N_1-3.
	\end{equation}
	According to \eqref{LruL}, one has
	\begin{equation*}
	L(r^{1/2}\underline L\bar Z^{k_0+1}\phi^I)=\f 1{2r^{1/2}}L\bar Z^{k_0+1}\phi^I+\f1{r^{3/2}}\O^2\bar Z^{k_0+1}\phi^I-\sum_{k_1+k_2\leq k_0+1}C^{k_0}_{k_1,k_2}r^{1/2}\tilde Q^I(\p\bar Z^{k_1}\phi,\p\bar Z^{k_2}\phi),
	\end{equation*}
	where $\tilde Q^I$ are also quadratic forms satisfying null conditions. We use \eqref{local3-3} and \eqref{LruL} to obtain
	$$
	|L\bar Z^{k_0+1}\phi|_{C_\dl}\lesssim\dl^{1-\kappa}t^{-1/2}.
	$$
	Similarly, with an induction argument on the number of $\underline L$, it follows from \eqref{LruL} that on $C_\dl$,
	$$
	|\underline L^m\bar Z^k\phi|\lesssim\dl^{1-\kappa}t^{-1/2},\quad 2m+k\leq 2N_1.
	$$
	Since $\p_i=\f{\o^i}2(\f{S+\o^j H_j}{2\underline u}-\underline L)+\o^i_{\perp}\f{\O}r$ and $\p_t=\f12(\f{S+\o^j H_j}{2\underline u}+\underline L)$, then
	\begin{equation}\label{Zkphi}
	|Z^k\phi|_{C_\dl}\lesssim\dl^{1-\kappa}t^{-1/2},\quad k\leq N_1.
	\end{equation}
	
	{
In addition, it follows from
\begin{equation*}
\begin{split}
&\underline L(\f1 {2r}Z^k\phi^I+LZ^k\phi^I)=\f 1{2r^2}Z^k\phi^I+\f1{2r}LZ^k\phi^I+\f1{r^2}\O^2Z^k\phi^I-Z^k\Box\phi^I-[\Box,Z^k]\phi^I,
\end{split}
\end{equation*}
\eqref{semi} and \eqref{Zkp} that on $D^0=\{(t,x): 0\leq t-|x|\leq 2\dl, t+|x|\geq 2+2\dl\}$,
\begin{equation}\label{LLa}
|\underline L(\f1 {2r}Z^k\phi^I+LZ^k\phi^I)|\lesssim\delta^{1-2\kappa-l}t^{-2},\ k\leq 2N_1-2,
\end{equation}
where $l$ is the number of $\p$ in $Z^k$.
Integrating \eqref{LLa} along integral curves of $\underline L$ in $D^0$ yields
\begin{equation}\label{La}
|\f1 {2r}Z^k\phi^I+LZ^k\phi^I|_{C_{\dl}}\lesssim\delta^{2-2\kappa-l}t^{-2},\ k\leq 2N_1-2.
\end{equation}
In particular, \eqref{La} implies $|L(r^{1/2}\O^k\phi^I)|_{ C_{\dl}}\lesssim\delta^{2-2\kappa}t^{-3/2}$, and then
\begin{equation}\label{OphiI}
|\O^k\phi^I|_{C_{\dl}}\lesssim\delta^{2-2\kappa}t^{-1/2}\ \text{for}\ k\leq 2N_1-2
\end{equation}
thanks to \eqref{local3-2}.
It follows from \eqref{La} and \eqref{OphiI} directly that
\begin{equation}\label{LO}
|L\O^k\phi^I|_{C_{\dl}}\lesssim\delta^{2-2\kappa}t^{-3/2}\ \text{for}\ k\leq 2N_1-2.
\end{equation}	
}
	
	\section{Global existence inside the outermost outgoing cone}\label{sec:interior}
	In this section, we prove the existence of the solution $\phi$ to \eqref{semi} with \eqref{initial} inside $B_{2\dl}:=\{(t,x)|t-|x|\geq 2\dl\}$.  Different from the small value problem inside $B_{2\dl}$, the solution
	$\phi$ to \eqref{semi} with \eqref{initial} in $B_{2\dl}$ remains large here due to its initial data on time $t_0$ (see
	Theorem \ref{Th2.1}). Note that for $\delta>0$ small, the $L^\infty$ norm of $\phi$ and
	its first order derivatives are small on the boundary $C_{\delta}$ of $B_{2\dl}$ (especially, $Z^{\al}\phi$ admits the better smallness
	$O(\dl^{1-\kappa})$ on $ C_{\delta}$, see \eqref{Zkphi}).

	Unlike \cite{Ding6} and \cite{Ding3}, we will not establish energy estimates on the hypersurface $\Sigma_t\cap B_{2\dl}$ since it seems hard to close the estimates there. Inspired by works \cite{Klainerman85, Hormander, LM, DLL}, we apply Klainerman's hyperboloidal method, i.e., we use hyperboloids to foliate the interior region and thus will bound the energy on hyperboloids. This method allows us to take full use of the $(t-|x|)$-decay of the solution, and hence close the energy estimates. But slightly different from previous works, we should construct a modified version of Klainermann-Sobolev inequality, (conformal) energy estimates, etc., which are applicable to our large data problem.

	\begin{figure}[htbp]
		\centering
		\subfigure[]{\includegraphics[width=0.5\linewidth]{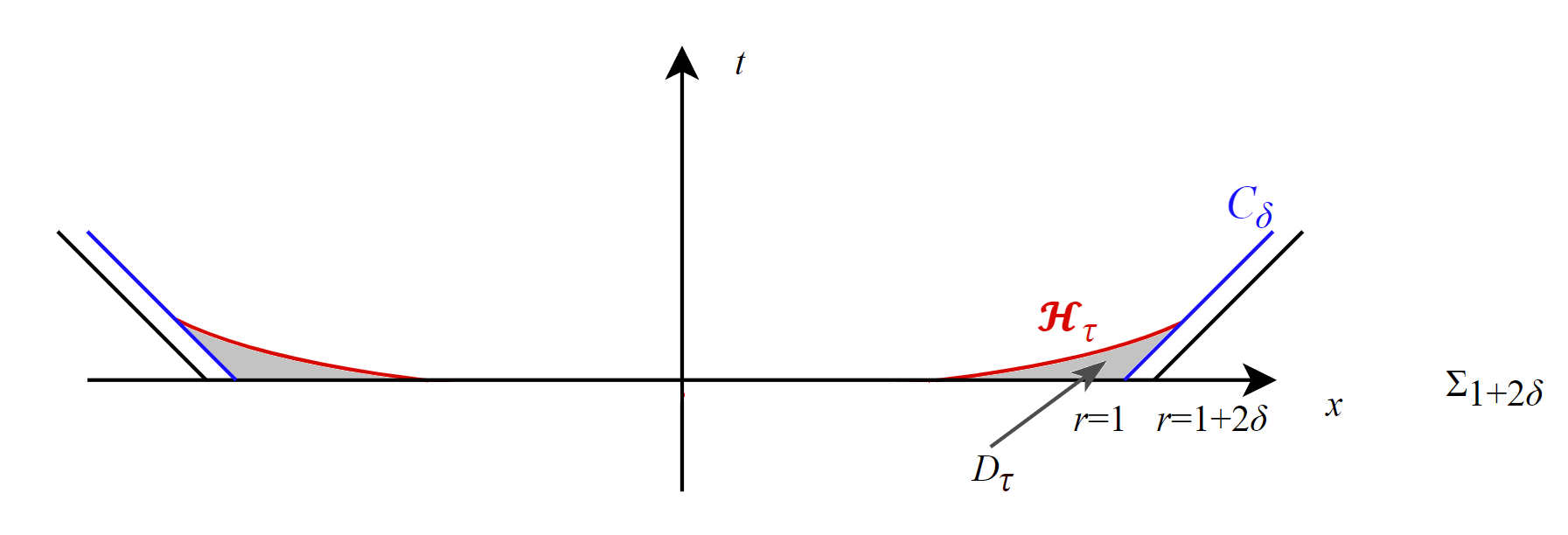}}\subfigure[]{\includegraphics[width=0.5\linewidth]{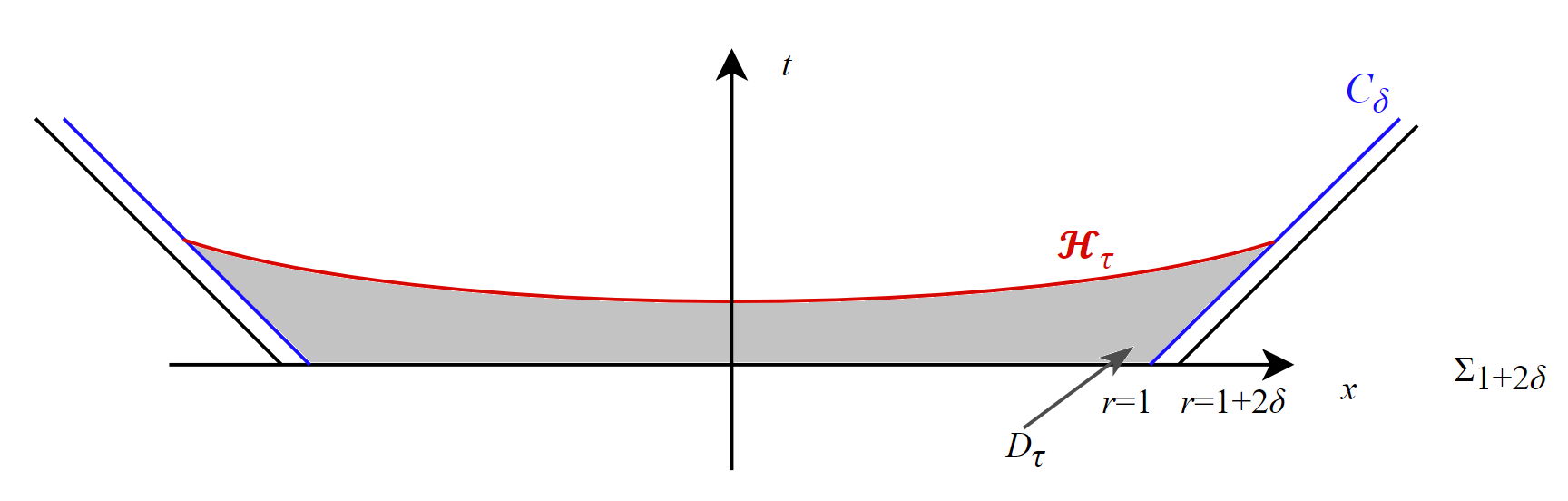}}
		\caption{ The domain $D_\tau$ inside $B_{2\dl}$}\label{pic:04}
	\end{figure}
	
	In this section, we perform the change of variables as follows:
	\begin{equation*}
	\left\{
	\begin{aligned}
	&\tau=\sqrt{(t+1)^2-|x|^2},\\
	&y_i=x_i,\quad i=1,2,
	\end{aligned}
	\right.
	\end{equation*}
	the hyperboloids are defined by
	\begin{equation*}
	\mathcal H_\tau:=\{(t,x)\in B_{2\dl}:(t+1)^2=|x|^2+\tau^2\},\quad \tau\geq \tau_0= \sqrt{3+4\dl^2+8\dl},
	\end{equation*}
	and $D_\tau$ is a subset of $B_{2\dl}$ which is surrounded by $\mathcal H_\tau$, $C_\dl$ and $\Sigma_{1+2\dl}$.
	Then, under new variables $(\tau,y_1,y_2)$,
	\begin{equation*}
	\bar\p_\tau=\bar\p_0=\f{\tau}{t+1}\p_t,\quad\bar\p_i=-\f{x_i}{t(t+1)}\p_t+\f1t H_i,\quad i=1,2.
	\end{equation*}
	
	For any smooth function $f(t,x)$, we define its $L^2$ norm  on $\mathcal H_\tau$ as
	\begin{equation*}
	\|f\|_{L^2(\mathcal H_\tau)}=\sqrt{\int_{\mathcal H_\tau}f^2(t,x)dx}:=\sqrt{\int_{\mathbb R^2}f^2(\sqrt{\tau^2+|y|^2}-1,y)dy},
	\end{equation*}
	and the natural energy and conformal energy  are defined as
	\begin{align}
	E(f,\tau)=&\|\p_{\perp}f\|^2_{L^2(\mathcal H_\tau)}+\sum_{i=1}^2\|\f{\tau}{t+1}\p_i f\|^2_{L^2(\mathcal H_\tau)}+\|\f1{t+1}\O f\|^2_{L^2(\mathcal H_\tau)}\no\\
	=&\|\bar\p_\tau f\|^2_{L^2(\mathcal H_\tau)}+\sum_{i=1}^2\|\bar\p_if\|^2_{L^2(\mathcal H_\tau)},\label{E(f,tau)}\\
	E_{c}(f,\tau)=&\|Kf+f\|^2_{L^2(\mathcal H_\tau)}+\sum_{i=1}^2\|s\bar\p_if\|^2_{L^2(\mathcal H_\tau)},\label{E_c(f,tau)}
	\end{align}
	where $\p_\perp=\p_t+\f{x^i}{t+1}\p_i$ and $K=\tau\bar\p_\tau+2y^i\bar\p_i$.
	
	The energy estimates for wave equations on hyperboloids are now well-known; see for instance \cite[eqn.  (7.6.8)]{Hormander}. But different from the small data and compact supported problems considered before, here we have non-zero  data on the boundary $C_\dl$, so we also provide a detailed  proof for completeness.
	\begin{lemma}\label{energy}
		Suppose that $\vp=\tilde\Gamma^k\O^l\phi^I$ with $\tilde\Gamma\in\{\p_{\al},S,H_1,H_2\}$, then we have the following.
		\begin{itemize}
			\item Natural energy estimates.
			\begin{align}
			\sqrt{E(\vp,\tau)}\lesssim&{\dl^{\f32-\kappa-k}}+\int_{\tau_0}^\tau\|\Box\vp\|_{L^2(\mathcal H_{\tilde \tau})}d{\tilde \tau}.\label{E}
			\end{align}
			
			\item Conformal energy estimates.
			\begin{align}
			\sqrt{E_c(\vp,\tau)}\lesssim&{\dl^{\f32-\kappa-k}\sqrt{\ln \tau}}+\int_{\tau_0}^\tau\tilde \tau\|\Box\vp\|_{L^2(\mathcal H_{\tilde \tau})}d\tilde \tau.\label{Ec}
			\end{align}
			
			\item $L^2$-type estimates.
			\begin{align}
			\|\f{\tau}{1+t}\vp\|_{L^2(\mathcal H_\tau)}\lesssim&{\dl^{\f32-\kappa-k}\sqrt{\ln \tau}}+\int_{\tau_0}^\tau\tilde \tau^{-1}\sqrt{E_c(\vp,\tilde \tau)}d\tilde \tau.\label{vp}
			\end{align}	
		\end{itemize}
		
	\end{lemma}
	\begin{proof}
		From the relationship between $(t,x_1,x_2)$ and $(\tau,y_1,y_2)$, one has
		\begin{equation*}
		\bar\p_\al=\bar\Phi_\al^\beta\p_\beta\quad\text{and}\quad\p_\al=\bar\Psi_\al^\beta\bar\p_\beta
		\end{equation*}
		with
		\begin{equation*}
		(\bar\Phi_\al^\beta)_{\al\beta}=\begin{pmatrix}
		\f{\tau}{t+1} & 0 &0\\
		\f{x^1}{t+1} &1 &0\\
		\f{x^2}{t+1} &0 &1
		\end{pmatrix}
		\quad\text{and}\quad
		(\bar\Psi_\al^\beta)_{\al\beta}=
		\begin{pmatrix}
		\f{t+1}\tau &0 &0\\
		-\f{y^1}\tau &1 &0\\
		-\f{y^2}\tau &0 &1
		\end{pmatrix}.
		\end{equation*}
		Thus, we have
		\begin{equation}\label{Box}
		\begin{split}
		\Box=&(m^{\al\beta}\bar\Psi_\al^{\al'}\bar\Psi_\beta^{\beta'})\bar\p_{\al'}\bar\p_{\beta'}+m^{\al\beta}(\p_\al\bar\Psi_\beta^{\beta'})\bar\p_{\beta'}\\
		=&-\bar\p_\tau^2-\f{2y^i}\tau\bar\p_\tau\bar\p_i+\sum_{i=1}^2\bar\p_i^2-\f2\tau\bar\p_\tau.
		\end{split}
		\end{equation}
		
		\textbf{Step 1: Proof of \eqref{E}.}	
		
		We apply \eqref{Box} to deduce
		\begin{equation}\label{mulitiplier}
		-\bar\p_\tau\vp\cdot\Box\vp=\f12\bar\p_\tau\big(|\bar\p_\tau\vp|^2+\sum_{i=1}^2|\bar\p_i\vp|^2\big)+\sum_{i=1}^2\bar\p_i\big(\f{y^i}\tau|\bar\p_\tau\vp|^2-\bar\p_\tau\vp\cdot\bar\p_i\vp\big),
		\end{equation}
		it follows from integrating \eqref{mulitiplier} over $D_\tau$ that
		\begin{equation}\label{Ds}
		\begin{split}
		&-2\int_{D_\tau}\bar\p_\tau\vp\cdot\Box\vp dyd\tilde \tau=E(\vp,\tau)-\int_{\Sigma_{t_0}\cap\overline{D_\tau}}\f{\tilde \tau}{t_0+1}\big(|\p_t\vp|^2+|\nabla\vp|^2\big)dx\\
		&{-\int_{C_\dl\cap\overline{D_\tau}}\f1{\sqrt{\tilde \tau^2+(2\dl+1)^2}}\big(\sum_{i=1}^2\frac{\tilde\tau}{t^2}(2\dl\o^i\p_t\vp+H_i\vp)^2\big)dS}.
		\end{split}
		\end{equation}
		On $C_\dl$, $\tilde \tau^2=(1+2\dl)(2t+1-2\dl)$, then $\tilde \tau\sim \sqrt t$, which means
		\begin{equation*}{
		|\p_t\vp|_{C_{\dl}}\lesssim\dl^{1-\kappa}\tilde \tau^{-1},\qquad
		|H_i\vp|_{C_\dl}\lesssim\left\{
		\begin{aligned}
		&\dl^{2-2\kappa}\tilde \tau^{-1},\qquad k=0,\\
		&\dl^{1-\kappa}\tilde \tau^{-1},\qquad k\geq 1
		\end{aligned}\right.}
\end{equation*}
		with the help of the fact $ H_i=\o^i\big(\f{r-t}{2}\underline L+\f{t+r}{2}L\big)+\f{t\o^i_\perp}{r}\O$, \eqref{Zkphi} and \eqref{OphiI}-\eqref{LO}, and therefore,
		\begin{equation}\label{Cdl}
		\begin{split}
		\int_{C_\dl\cap\overline{D_\tau}}\f1{\sqrt{\tilde \tau^2+(2\dl+1)^2}}\big(\sum_{i=1}^2\frac{\tilde\tau}{t^2}(2\dl\o^i\p_t\vp+H_i\vp)^2\big)dS
		\lesssim{\dl^{4-4\kappa-2k}}.
		\end{split}
		\end{equation}
		Substitute \eqref{Cdl} to \eqref{Ds} and then one obtains
		\begin{equation*}
		\begin{split}
		E(\vp,\tau)\lesssim\int_{\Sigma_{t_0}\cap\overline{D_\tau}}|\p\vp(t_0,x)|^2dx+{\dl^{4-4\kappa-2k}}+\int_{\tau_0}^\tau\sqrt{E(\vp,\tilde \tau)}\|\Box\vp\|_{L^2(\mathcal H_{\tilde \tau})}d\tilde \tau,
		\end{split}
		\end{equation*}
		which proves \eqref{E} by Gronwall's inequality and \eqref{local3-2} since $\kappa\in(0,\f12)$.
		
		\textbf{Step 2: Proof of \eqref{Ec}.}
		
		According to the definition of $K$ in \eqref{E_c(f,tau)} and the identity \eqref{Box}, one has
		\begin{align*}
		-\tau K\vp\cdot\Box\vp=&\f12\bar\p_\tau(|K\vp|^2)-\sum_{i=1}^2\bar\p_i\big(\tau^2(\bar\p_\tau\vp+\f{2y^j}\tau\bar\p_j\vp)\bar\p_i\vp\big)+\f12\sum_{i=1}^2\bar\p_\tau(|\tau\bar\p_i\vp|^2)\\
		&+\tau\big(\bar\p_\tau\vp(\bar\p_\tau\vp+\f{2y^j}\tau\bar\p_j\vp)-\sum_{i=1}^2|\bar\p_i\vp|^2\big)+\sum_{j=1}^2\bar\p_i(\tau y^i|\bar\p_j\vp|^2),\\
		-\tau\vp\Box\vp=&\bar\p_\tau(\tau\vp\bar\p_\tau\vp)+\f12\bar\p_\tau(\vp^2)+\bar\p_\tau(2y^i\vp\bar\p_i\vp)-\sum_{i=1}^2\bar\p_i(\tau\vp\bar\p_i\vp)\\
		&-\tau\big((\bar\p_\tau\vp)^2+\f{2y^i}\tau(\bar\p_\tau\vp)(\bar\p_i\vp)-\sum_{i=1}^2(\bar\p_i\vp)^2\big),
		\end{align*}
		and hence,
		\begin{equation}\label{D}
		\begin{split}
		-2\tau (K\vp+\vp)\Box\vp=&\bar\p_\tau\big((K\vp+\vp)^2+\sum_{i=1}^2|\tau\bar\p_i\vp|^2\big)\\
		&+\sum_{i=1}^2\bar\p_i\big(-2\tau (K\vp+\vp)\bar\p_i\vp+2\tau y^i\sum_{j=1}^2|\bar\p_j\vp|^2\big).
		\end{split}
		\end{equation}
		If follows from integrating \eqref{D} over $D_\tau$ that
		\begin{equation}\label{Ec(vp)}
		\begin{split}
		-2\int_{D_\tau}&\tilde \tau(K\vp+\vp)\Box\vp dyd\tilde \tau=E_c(\vp,\tau)\\
		&{-\int_{C_\dl\cap\overline{D_\tau}}\f{\tilde \tau\sum_{i=1}^2\big(\o^i(K\vp+\vp)+(1+2\dl)\bar\p_i\vp\big)^2}{\sqrt{\tilde \tau^2+(2\dl+1)^2}}dS}\\
		&-\int_{\Sigma_{t_0}\cap\overline{D_\tau}}\f{\tilde \tau\big((K\vp+\vp)^2+\sum_{i=1}^2{(\tilde \tau^2+2r^2)|\bar\p_i\vp|^2}-2(K\vp+\vp) y^i\bar\p_i\vp\big)}{t_0+1}dy
		\end{split}
		\end{equation}
		{It follows from $K\vp=\f{(1+2\dl)^2}{1+t}\p_t\vp+2rL\vp$ that on $C_\dl$,
		\begin{equation*}
		\begin{split}
			\o^i(K\vp+\vp)+(1+2\dl)\bar\p_i\vp=&\o^i(2rL\vp+\vp)+\f{2\dl(1+2\dl)}t\o^i\p_t\vp\\
			&+\f{1+2\dl}t\o^i(-\dl\underline L\vp+\underline uL\vp)+\f{1+2\dl}r\o^i_{\perp}\O\vp
		\end{split}
		\end{equation*}
		which implies
		\begin{equation*}
		\int_{C_\dl\cap\overline{D_\tau}}\f{\tilde \tau\sum_{i=1}^2\big(\o^i(K\vp+\vp)+(1+2\dl)\bar\p_i\vp\big)^2}{\sqrt{\tilde \tau^2+(2\dl+1)^2}}dS
		\lesssim\dl^{4-4\kappa-2k}\ln \tau
		\end{equation*}
		with the help of \eqref{Zkphi} and \eqref{La}-\eqref{LO}.}
		Therefore,
		\begin{equation*}
		\begin{split}
		E_c(\vp,\tau)\lesssim&\dl^{4-4\kappa-2k}\ln \tau+\int_{\Sigma_{t_0}\cap\overline{D_\tau}}(|\p\vp(t_0,x)|^2+|\vp(t_0,x)|^2)dx\\
		&+\int_{\tau_0}^\tau\tilde \tau\sqrt{E_c(\vp,\tilde \tau)}\|\Box\vp\|_{L^2(\mathcal H_{\tilde \tau})}d\tilde \tau
		\end{split}
		\end{equation*}
		by \eqref{Ec(vp)}, and then \eqref{Ec} holds.
		
		\textbf{Step 3: Proof of \eqref{vp}.}	
		
		Finally, we prove \eqref{vp}, which was observed in \cite{M} and our proof also follows the proof there. By the identity
		\begin{equation*}
		\f{2\tau}{t+1}\vp\cdot\f{K\vp+\vp}{t+1}=\bar\p_\tau\big(\f{\tau^2}{(t+1)^2}\vp^2\big)+\bar\p_i\big(\f{y^i\tau }{(t+1)^2}\vp^2\big)+\f{2}{t+1}\vp\big(\f{y^i\tau}{t+1}\bar\p_i\vp\big),
		\end{equation*}
		we have
		\begin{equation*}
		\begin{split}
		&2\int_{D_\tau}\f{\tau}{t+1}\vp\cdot\f{K\vp+\vp}{t+1}dxd\tilde \tau-2\int_{D_\tau}\f{\t\tau}{t+1}\vp\big(\f{y^i}{t+1}\bar\p_i\vp\big)dyd\tilde \tau\\
		=&\|\f{\tau}{t+1}\vp^2\|_{L^2(\mathcal H_\tau)}-\int_{C_\dl\cap\overline{D_\tau}}\f{\tilde \tau(1+2\dl)}{(t+1)\sqrt{\tilde \tau^2+(1+2\dl)^2}}\vp^2dS-\int_{\Sigma_{t_0}\cap\overline{D_\tau}}\f{\tilde \tau}{t_0+1}\vp^2dx,
		\end{split}
		\end{equation*}
		where
		\begin{equation*}
		\int_{C_\dl\cap\overline{D_\tau}}\f{\tilde \tau(1+2\dl)}{(t+1)\sqrt{\tilde \tau^2+(1+2\dl)^2}}\vp^2dS\lesssim{\dl^{4-4\kappa-2k}\ln \tau}.
		\end{equation*}
		Thus,
		\begin{equation*}
		\|\f{\tau}{t+1}\vp\|^2_{L^2(\mathcal H_\tau)}\lesssim\dl^{4-4\kappa-2k}\ln \tau+\int_{\Sigma_{t_0}\cap\overline{D_\tau}}\vp(t_0,x)^2dx+\int_{\tau_0}^\tau\tilde \tau^{-1}\|\f{\tilde \tau}{t+1}\vp\|_{L^2(\mathcal H_{\tilde \tau})}\sqrt{E_c(\vp,\tilde \tau)}d\tilde \tau,
		\end{equation*}
		and \eqref{vp} is proved.
	\end{proof}
	
	Form \eqref{vp}, we know that the norm of $\f{\tau}{1+t}\vp$ can be estimate by the initial data and conformal energy $E_c(\vp,\tau)$. Moreover, $\|\f \tau{t+1}S\vp\|_{L^2(\mathcal H_\tau)}$ and $\|\f \tau{t+1}H_i\vp\|_{L^2(\mathcal H_\tau)}$ are also controlled by $E(\vp,\tau)$ and $E_c(\vp,\tau)$.
	\begin{lemma}
		For any smooth function $\vp$ in $D_\tau$, one has
		\begin{equation}\label{SHvp}
		\|\f \tau{t+1}S\vp\|_{L^2(\mathcal H_\tau)}+\sum_{i=1}^2\|\f \tau{t+1}H_i\vp\|_{L^2(\mathcal H_\tau)}\lesssim\|\f \tau{t+1}\vp\|_{L^2(\mathcal H_\tau)}+\sqrt{E_c(\vp,\tau)}+\sqrt{E(\vp,\tau)}.
		\end{equation}
	\end{lemma}
	\begin{proof}
		\eqref{SHvp} follows directly from the identities
		\begin{align*}
		&\f \tau tH_i\vp=\tau\bar\p_i\vp+\f{y^i}{t}\bar\p_\tau\vp,\\
		&\f{\tau}{t+1}S\vp=\f \tau{t+1}(K\vp+\vp)-\f \tau{t+1}\vp-\f{t+1-\tau^2}{t(t+1)}\bar\p_\tau\vp-\f{\tau y^i}{t(t+1)}H_i\vp.
		\end{align*}
	\end{proof}

	We next state the Klainerman-Sobolev inequality on hyperboloids in our setting, in which we distinguish the spacetime regions $\{r\leq t/4 \}$ and $\{r\geq t/4\}$.
	
	\begin{lemma}[Klainerman-Sobolev inequality]\label{KS}
		Suppose $\vp$ is the same function as in Lemma \ref{energy},  $(\bar t,\bar x)\in \mathcal H_{\tau}$, the following inequalities hold:
		\begin{equation}\label{leq}
		\mid \vp(\bar t,\bar x)\mid\lesssim \tau^{-1}\sum_{i=0}^2 \delta^{(i-1)\lambda}\|\f{\tau}{1+t}{\Gamma}^{i}\vp\|_{L^2(\mathcal H_\tau)},\ |\bar x|\leq\f14\bar t,
		\end{equation}
		\begin{equation}\label{geq}
		\mid \vp(\bar t,\bar x)\mid\lesssim \dl^{1-\kappa}\tau^{-1}+\sum_{a\leq 1,b\leq 1}\tau^{-1}
		\|\f \tau{1+t}\Omega^a\vp\|^{1/2}_{L^2(\mathcal H_\tau)}\|\f \tau{1+t}\O^b\Gamma\vp\|^{1/2}_{L^2(\mathcal H_\tau)},\ |\bar x|\geq\f14\bar t,
		\end{equation}
		where ${\Gamma}\in\{\p_t,H_i\}$
		and $\lambda$ is  any nonnegative constant in \eqref{leq}.
	\end{lemma}
	\begin{proof}
		\textbf{Step 1: Proof of \eqref{leq}.}	The proof in this part follows from \cite[Lem. 7.6.1]{Hormander} and \cite{Ding4}.

		Let $\chi\in C_c^\infty(\mathbb R_+)$ be a nonnegative cut-off function such that $\chi(r)\equiv1$ for $r\in[0,\f14]$ and $\chi(r)\equiv0$ for $r\geq\f12$, set $$w_\tau(x)=\vp(\sqrt{\tau^2+|x|^2}-1,x)\chi(\f{|x|}{\sqrt{\tau^2+|x|^2}-1}).$$
		
		First, for any point $(\bar t,\bar x)\in\mathcal H_\tau$ satisfying $|\bar x|\leq\f14\bar t$, one has $w_\tau(\bar x)=\vp(\bar t,\bar x)$. For any constant $\lambda\geq 0$ and $z\in\mathbb R^2$, denote
		$$
		g_{\tau,\bar t}(z)=w_\tau(\bar x+\bar t\dl^\lambda z),
		$$
		then $g_{\tau,\bar t}(0)=w_\tau(\bar x)=\vp(\bar t,\bar x)$.
		It follows from the Sobolev embedding theorem for $z$ that
		\begin{equation}\label{ins}
		\begin{split}
		|\vp(\bar t,\bar x)|^2=&|g_{\tau,\bar t}(0)|^2\lesssim\sum_{|\al|\leq 2}\int_{|z|\leq\f12}|\p_z^\al
		g_{\tau,\bar t}(z)|^2dz\\
		\lesssim&\int_{\{|x|\leq \f12t\}}|\vp(t,x)|dz+\sum_{i=1}^2\int_{\{|x|\leq\f12t\}\cap\{|z|\leq\f12\}}(\bar t\dl^\lambda)^2|\bar\p_i\vp(t,x)|^2dz\\
		&+\sum_{i,j=1}^2\int_{\{|x|\leq\f12t\}\cap\{|z|\leq\f12\}}(\bar t\dl^\lambda)^4|\bar\p_{ij}^2\vp(t,x)|^2dz\\
		\lesssim&\int_{\{|x|\leq \f12t\}}|\vp(t,x)|^2(\bar t\dl^\lambda)^{-2}dx+\sum_{i=1}^2\int_{\{|x|\leq\f12t\}\cap\{|x-\bar x|\leq\bar t\dl^{\lambda}/2\}}|\bar\p_i\vp(t,x)|^2dx\\
		&+\sum_{i,j=1}^2\int_{\{|x|\leq\f12t\}\cap\{|x-\bar x|\leq\bar t\dl^{\lambda}/2\}}|\bar\p_{ij}^2\vp(t,x)|^2(\bar t\dl^\lambda)^2dx,
		\end{split}
		\end{equation}
		where $x=\bar x+\bar t\dl^\lambda z$ and $t=\sqrt{\tau^2+|x|^2}-1$.
		Note that when $|x|\leq\f12t$ and $|x-\bar x|\leq\bar t\dl^\lambda/2$,
		\begin{equation}\label{tf}
		\begin{split}
		&\bar t\lesssim t,\\
		&|t\bar\p_i\vp|\leq|\p_t\vp|+|H_i\vp|,\\
		&|t^2\bar\p_{ij}^2\vp|\lesssim|\p_t\vp|+|\p_t^2\vp|+\sum_{k=1}^2(|\p_tH_k\vp|+|H_k\vp|)+\sum_{k,l=1}^2|H_kH_l\vp|.
		\end{split}
		\end{equation}
		Therefore, substituting \eqref{tf} into \eqref{ins} yields
		\begin{equation*}
		\begin{split}
		|\vp(\bar t,\bar x)|^2\lesssim\sum_{i=0}^2\bar t^{-2}\dl^{2(i-1)\lambda}\int_{\mathcal H_\tau\cap\{|x|\leq\f12t\}}|\G^i\vp(t,x)|^2dx,
		\end{split}
		\end{equation*}
		then,
		\begin{equation}\label{stvp}
		\begin{split}
		|\f \tau{\bar t+1}\vp(\bar t,\bar x)|^2\lesssim\sum_{i=0}^2\bar t^{-2}\dl^{2(i-1)\lambda}\int_{\mathcal H_\tau\cap\{|x|\leq\f12t\}}|\G^i\big(\f{\tau}{t+1}\vp(t,x)\big)|^2dx.
		\end{split}
		\end{equation}
		Since when $(t,x)\in\mathcal H_\tau$ and $|x|\leq\f12t$, $|x|\lesssim \tau$ and $t+1\lesssim \tau$, then
		\begin{align*}
		&|\G(\f \tau{t+1}\vp(t,x))|\lesssim|\f \tau{1+t}\vp(t,x)|+|\f \tau{t+1}\G\vp(t,x)|,\\
		&|\G^2(\f \tau{t+1}\vp(t,x))|\lesssim|\f \tau{1+t}\vp(t,x)|+|\f \tau{t+1}\G\vp(t,x)|+|\f \tau{t+1}\G^2\vp(t,x)|,
		\end{align*}
		\eqref{leq} follows when taking the above inequalities to \eqref{stvp}.
		
		\textbf{Step 2: Proof of \eqref{leq}.}
		
		For $(\bar t,\bar x)\in\mathcal H_\tau$ satisfying $|\bar x|\geq\f14\bar t$, let $\bar x=|\bar x|\o$ and $B_{\bar t,\bar x}=(\t t,\t x)$ be the point on $C_\dl\cap\mathcal H_\tau$ whose angular component is $\o$, that is, $\t x=|\t x|\o$, $\t t=|\t x|+2\dl$ and $|\t x|=\f{\tau^2-(1+2\dl)^2}{2(1+2\dl)}$. By the Newton-Leibnitz formula, one has
		\begin{equation*}
		\begin{split}
		\vp^2(\sqrt{\tau^2+|\bar x|^2}-1,\bar x)&=\vp^2(B_{\bar t,\bar x})-\int_{|\bar x|}^{|\t x|}\p_\rho\big(\vp^2(\sqrt{\tau^2+\rho^2}-1,\rho\omega)\big)d\rho\\
		&=\vp^2(B_{\bar t,\bar x})-2\int_{|\bar x|}^{|\t x|}[\vp\o^i\bar\p_i\vp](\sqrt{\tau^2+\rho^2}-1,\rho\omega)d\rho.
		\end{split}
		\end{equation*}
		Since
		$$
		|\bar\p_i\vp(t,\rho\omega)|\leq\f1t(|\p_t\vp(t,\rho\omega)|+|H_i\vp(t,\rho\omega)|)
		$$
		with $t=\sqrt{\tau^2+\rho^2}-1$, then it follows from \eqref{Zkphi}, Sobolev embedding theorem on the circle and $\f\rho t\geq\f14$ that
		\begin{equation*}
		\begin{split}
		|\vp(\bar t,\bar x)|^2\lesssim&\dl^{2-2\kappa}\tau^{-2}+\int_{|\bar x|}^{|\t x|}\f1t\big(|\vp|(|\p_t\vp|+\sum_{i=1}^2|H_i\vp|)\big)(\sqrt{\tau^2+\rho^2}-1,\rho\o)d\rho\\
		\lesssim&\dl^{2-2\kappa}\tau^{-2}+\int_{|\bar x|}^{|\t x|}\f1{t\rho}\|\O^{\leq 1}\vp\|_{L^2(S_t^\rho)}\big(\|\O^{\leq 1}\p_t\vp\|_{L^2(S_t^\rho)}+\sum_{i=1}^2\|\O^{\leq 1}H_i\vp\|_{L^2(S_t^\rho)}\big)d\rho\\
		\lesssim&\dl^{2-2\kappa}\tau^{-2}+\|\f1t\O^{\leq 1}\vp\|_{L^2(\mathcal H_\tau)}\big(\|\f1t\O^{\leq 1}\p_t\vp\|_{L^2(\mathcal H_\tau)}+\sum_{i=1}^2\|\f1t\O^{\leq 1}H_i\vp\|_{L^2(\mathcal H_\tau)}\big),
		\end{split}
		\end{equation*}
		where $S_t^\rho$ is a circle with radius $\rho$ and center at $(t,0)$,
		\eqref{geq} is established. Thus Lemma \ref{KS} is verified.
	\end{proof}

	We will apply the energy method to prove the global existence of solution $\phi$ to \eqref{semi} with \eqref{initial} in $B_{2\dl}$. Unlike establishing energy on the hypersurface $\Sigma_t\cap B_{2\dl}$ in \cite{Ding3}, we perform our energy estimate on $\mathcal H_\tau$ since it admits more higher decay rate which could help us to close our bootstrap assumption. Define
	\begin{align}
	E_{k,l}(\tau)=&\sum_{\t\Gamma\in\{\p,H_i,S\}}\sum_{1\leq I\leq N}E(\tilde\Gamma^k\O^l\phi^I,\tau),\\
	E_{k,l}^c(\tau)=&\sum_{\t\Gamma\in\{\p,H_i,S\}}\sum_{1\leq I\leq N}E_c(\tilde\Gamma^k\O^l\phi^I,\tau),
	\end{align}
	based on \eqref{local3-2}, one can make the following bootstrap assumption:
	For $\tau\ge \tau_0$, there exists a uniform constant $M_0$ such that
	\begin{equation}\label{EA}
	\begin{split}
	E_{k,l}(\tau)\leq {M_0}^2\delta^{2a_k},\quad\text{and}\quad E^c_{k,l}(\tau)\leq {M_0}^2\delta^{2a_k}{\tau^{2\eta}},\quad k+l\leq 4
	\end{split}
	\end{equation}
	with {$a_k=\f32-\kappa-k$ and $\eta$ is a fixed number in $(0,\f1{10})$}.

	\begin{proposition}\label{P4.1}
		Under the assumptions \eqref{EA}, when $\delta>0$ is small and $k+l\leq 4$, it holds that
		\begin{equation}
		\begin{split}
		&\|\f \tau{1+t}\p\tilde\Gamma^k\O^l\phi^I\|_{L^2(\mathcal H_\tau)}\lesssim M_0\dl^{a_k},\\
		&\|\f{\tau}{1+t}\O\tilde\Gamma^k\O^l\phi^I\|_{L^2(\mathcal H_\tau)}+\|\f{\tau}{1+t}S\tilde\Gamma^k\O^l\phi^I\|_{L^2(\mathcal H_\tau)}+\sum_{i=1}^2\|\f{\tau}{1+t}H_i\tilde\Gamma^k\O^l\phi^I\|_{L^2(\mathcal H_\tau)}\lesssim  M_0\dl^{a_k}{\tau^{\eta}},
		\end{split}
		\end{equation}
		and when $k+l\leq 2$,
		\begin{equation}\label{tZ}
		\begin{split}
		&||\p\tilde\Gamma^k\O^l\phi^I||_{L^{\infty}(\mathcal H_\tau)}\lesssim M_0\dl^{\al_k}\tau^{-1},\\
		&||\O\tilde\Gamma^k\O^l\phi^I||_{L^{\infty}(\mathcal H_\tau)}+\sum_{i=1}^2||H_i\tilde\Gamma^k\O^l\phi^I||_{L^{\infty}(\mathcal H_\tau)}+||S\tilde\Gamma^k\O^l\phi^I||_{L^{\infty}(\mathcal H_\tau)}\lesssim M_0\dl^{\al_k}{\tau^{-1+\eta}},
		\end{split}
		\end{equation}
		where {$\al_k=\f12-\kappa-k$}.
	\end{proposition}	
	\begin{proof}
		According to the definition of $E(f,\tau)$ in \eqref{E(f,tau)}, one has
		\begin{equation}\label{pGO}
		\|\f \tau{1+t}\p\tilde\Gamma^k\O^l\phi^I\|_{L^2(\mathcal H_\tau)}\lesssim\sqrt{E(\tilde\Gamma^k\O^l\phi^I,\tau)}\lesssim M_0\dl^{a_k},\quad k+l\leq 4.
		\end{equation}
		And it follows from \eqref{vp} that when $k+l\leq 4$,
		\begin{equation*}
		\begin{split}
		\|\f \tau{t+1}\tilde\Gamma^k\O^l\phi^I\|_{L^2(\mathcal H_\tau)}\lesssim&\dl^{\f32-\kappa-k}\sqrt{\ln \tau}+\int_{\tau_0}^\tau\tilde \tau^{-1}\sqrt{E^c_{k,l}(\tilde \tau)}d\tilde \tau\\
		\lesssim&M_0\dl^{a_k}\tau^{\eta},
		\end{split}
		\end{equation*}
		this, together with \eqref{SHvp} and the fact $|\O f|\lesssim\sum_{i=1}^2|H_i f|$  in $D_\tau$ for any smooth function $f$, one has
		\begin{equation}\label{GGO}
		\begin{split}
		&\|\f{\tau}{1+t}\O\tilde\Gamma^k\O^l\phi^I\|_{L^2(\mathcal H_\tau)}+\|\f{\tau}{1+t}S\tilde\Gamma^k\O^l\phi^I\|_{L^2(\mathcal H_\tau)}+\sum_{i=1}^2\|\f{\tau}{1+t}H_i\tilde\Gamma^k\O^l\phi^I\|_{L^2(\mathcal H_\tau)}\\
		\lesssim&\|\f \tau{1+t}\tilde\Gamma^k\O^l\phi^I\|_{L^2(\mathcal H_\tau)}+\sqrt{E^c_{k,l}(\tau)}+\sqrt{E_{k,l}(\tau)}\\
		\lesssim& M_0\dl^{a_k}\tau^{\eta}
		\end{split}
		\end{equation}
		when $k+l\leq 4$.
		
		For any point $(\bar t,\bar x)\in\mathcal H_\tau$ satisfying $|\bar x|\leq\f {\bar t}4$ and $k+l\leq 2$, one gets from \eqref{leq} and \eqref{pGO} that
		\begin{equation}\label{Y-30}
		\begin{split}
		|\p\t\G^k\Omega^{l}\phi^I(\bar t,\bar x)|
		\lesssim& \tau^{-1}\sum_{i=0}^2\dl^{(i-1)\lambda_k}\|\f \tau{t+1}\G^i\p\t\G^k\Omega^{l}\phi^I\|_{L^2(\mathcal H_\tau)}\\
		\lesssim& \tau^{-1}\sum_{i=0}^2\dl^{(i-1)\lambda_k}\|\f \tau{t+1}\p\G^i\t\G^k\Omega^{l}\phi^I\|_{L^2(\mathcal H_\tau)}\\
		\lesssim& M_0\tau^{-1}(\dl^{a_k-\lambda_k}+\dl^{a_{k+1}}+\dl^{a_{k+2}+\lambda_k}),
		\end{split}
		\end{equation}
		where $\lambda_k$ is any nonnegative constant. Similarly, it follows from \eqref{leq} and \eqref{GGO} that
		\begin{equation}\label{OSH}
		\begin{split}
		&|\O\t\G^k\O^l\phi^I(\bar t,\bar x)|+|S\t\G^k\O^l\phi^I(\bar t,\bar x)|+\sum_{i=1}^2|H_i\t\G^k\O^l\phi^I(\bar t,\bar x)|\\
		\lesssim& M_0\tau^{-1+\eta}(\dl^{a_k-\lambda_k}+\dl^{a_{k+1}}+\dl^{a_{k+2}+\lambda_k}).
		\end{split}
		\end{equation}
		For any point $(\bar t,\bar x)\in\mathcal H_\tau$ satisfying $|\bar x|\geq\f {\bar t}4$ and $k+l\leq 2$, use \eqref{geq}, \eqref{pGO} and \eqref{GGO} to have
		\begin{align}
		&|\p\t\G^k\Omega^{l}\phi^I(\bar t,\bar x)|\lesssim M_0\tau^{-1}\dl^{(a_k+a_{k+1})/2},\label{ge}\\
		&|\O\t\G^k\O^l\phi^I(\bar t,\bar x)|+|S\t\G^k\O^l\phi^I(\bar t,\bar x)|+\sum_{i=1}^2|H_i\t\G^k\O^l\phi^I(\bar t,\bar x)|
		\lesssim M_0\tau^{-1+\eta}\dl^{(a_k+a_{k+1})/2}.\label{g}
		\end{align}
		Thus, \eqref{tZ} holds when we choose {$\lambda_k=1$} in \eqref{Y-30} and \eqref{OSH}.
	\end{proof}

	We are ready to close our bootstrap assumption \eqref{EA} in $D_\tau$, and hence obtain the following existence theorem in $B_{2\dl}$.
	
	\begin{theorem}\label{11.2}
		When $\delta>0$ is small, there exists a smooth solution $\phi$ to \eqref{semi} in $B_{2\dl}$.
	\end{theorem}
	\begin{proof}
		Act $\t\Gamma^k\O^l$ ($k+l\leq 4$) on the equation of \eqref{semi}, and then one has
		\begin{equation*}
		\begin{split}
		|\Box\t\G^k\O^l\phi^I|\lesssim(1+t)^{-1}\sum_{\tiny\begin{array}{c}k_1+k_2\leq k\\l_1+l_2\leq l\end{array}}|\p\t\G^{k_1}\O^{l_1}\phi|\cdot|Z\t\G^{k_2}\O^{l_2}\phi|,
		\end{split}
		\end{equation*}
		where $Z\in\{\p_{\al},S,H_i,\O\}$. Thus, it follows from Proposition \ref{P4.1} that
		\begin{equation}\label{key}
		\begin{split}
		\|\Box\t\G^k\O^l\phi^I\|_{L^2(\mathcal H_\tau)}\lesssim&\sum_{\tiny\begin{array}{c}k_1+k_2\leq k\\l_1+l_2\leq l\\k_1+l_1\geq k_2+l_2\end{array}}\tau^{-1}\|\f \tau{1+t}\p\tilde\G^{k_1}\O^{l_1}\phi\|_{L^2(\mathcal H_\tau)}\|Z\t\G^{k_2}\O^{l_2}\phi\|_{L^\infty(\mathcal H_\tau)}\\
		&+\sum_{\tiny\begin{array}{c}k_1+k_2\leq k\\l_1+l_2\leq l\\k_1+l_1> k_2+l_2\end{array}}\tau^{-1}\|\p\tilde\G^{k_1}\O^{l_1}\phi\|_{L^\infty(\mathcal H_\tau)}\|\f \tau{1+t}Z\t\G^{k_2}\O^{l_2}\phi\|_{L^2(\mathcal H_\tau)}\\
		\lesssim&\sum_{\tiny\begin{array}{c}k_1+k_2\leq k\\k_2\leq 2\end{array}}{M_0}^2\dl^{a_{k_1}+\al_{k_2}}{\tau^{-2+\eta}}.
		\end{split}
		\end{equation}
		Inserting \eqref{key} to \eqref{E} and \eqref{Ec}, one gets for $k+l\leq 4$,
		\begin{align}
		&\sqrt{E_{k,l}(\tau)}\lesssim\dl^{\f32-\kappa-k}+\sum_{\tiny\begin{array}{c}k_1+k_2\leq k\\k_2\leq 2\end{array}}{M_0}^2\dl^{a_{k_1}+\al_{k_2}},\label{Ekl}\\
		&\sqrt{E^c_{k,l}(\tau)}\lesssim\dl^{\f32-\kappa-k}\sqrt{\ln\tau}+\sum_{\tiny\begin{array}{c}k_1+k_2\leq k\\k_2\leq 2\end{array}}{M_0}^2\dl^{a_{k_1}+\al_{k_2}}\tau^{\eta}.\label{Eckl}
		\end{align}
		Take the value of $a_{k_1}$ and $\al_{k_2}$ in \eqref{EA} and \eqref{tZ} to \eqref{Ekl} and \eqref{Eckl}, then since $\kappa<\f12$ and $\dl>0$ is small enough, one has
		$$
		E_{k,l}(\tau)\lesssim\dl^{2a_k}\quad\text{and}\quad E^c_{k,l}(\tau)\lesssim\dl^{2a_k}\tau^{2\eta} \quad\text{for}\quad k+l\leq 4,
		$$
		which are independent of $M_0$.
	\end{proof}
	
	Finally, we prove Theorem \ref{main}.
	\begin{proof}
		Theorem \ref{Th2.1} gives the local existence of smooth solution $\phi$ to \eqref{semi} with \eqref{initial}.
		On the other hand,  the global existence of the solution in $A_{2\dl}$
		and in $B_{2\dl}$  has been established in Section \ref{YY} and Theorem \ref{11.2} respectively.
		Then it follows from the uniqueness of the smooth solution to \eqref{semi} with \eqref{initial}
		that the proof of $\phi\in C^\infty([1,+\infty)\times\mathbb R^2)$  is finished.
		In addition, $|\na\phi|\lesssim\delta^{-\kappa}t^{-1/2}$ follows from  \eqref{local1-2}, \eqref{local2-2},
		\eqref{Zkphi}, and the first inequality in \eqref{tZ} since $t\lesssim \tau^2$ in $B_{2\dl}$. Thus Theorem \ref{main}
		is proved.
	\end{proof}
	
	{\bf Acknowledgement.} Ding is supported by National Natural Science Foundation of China 12071223. Xu is supported by National Natural Science Foundation of China 11971237.

\end{document}